\documentclass[sort&compress]{elsarticle}
\usepackage{csl}

\usepackage{lineno,hyperref}
\modulolinenumbers[5]

\usepackage{amssymb,amsmath}
\usepackage{amsthm} 
\usepackage{bbm}
\usepackage{comment}
\usepackage{graphicx}
\usepackage{mleftright}
\usepackage{pgfplots}
\usepackage{physics}
\usepackage{rotating}
\usepackage{subcaption}
\usepackage{siunitx}

\usepackage{cleveref}
\pgfplotsset{compat=newest}
\ifpdf
  \DeclareGraphicsExtensions{.eps,.pdf,.png,.jpg}
\else
  \DeclareGraphicsExtensions{.eps}
\fi

\def\*#1{\mathbf{#1}}
\def\^#1{\widehat{#1}}
\def\-#1{\overline{#1}}
\def\~#1{\widetilde{#1}}
\def\>#1{\overrightarrow{#1}}

\DeclareMathOperator{\blkdiag}{blkdiag}

\newcommand{\A}{\mathbf{A}}

\newcommand{\B}{\mathbf{B}}
\renewcommand{\c}{\mathbf{c}}
\newcommand{\U}{\mathbf{U}}
\newcommand{\V}{\mathbf{V}}
\newcommand{\W}{\mathbf{W}}
\newcommand{\0}{\mathbf{0}}

\newcommand{\one}{\text{\usefont{U}{bbm}{m}{n}1}} 
\newcommand{\comp}[1]{^{\{#1\}}}
\newcommand{\tcomp}[1]{^{[#1]}}
\newcommand{\super}[2]{^{\{#1\}#2}}
\newcommand{\bigo}[1]{\mathcal{O}{\left(#1\right)}}
\newcommand{\method}[1]{ADI-DIMSIM#1}
\newtheorem{theorem}{Theorem}

\newtheorem{remark}{Remark}
\newtheorem{definition}{Definition}

\newenvironment{butchertableau}[2][1.4]
{\def\arraystretch{#1}\array{#2}}
{\endarray}
\journal{Journal of Computational and Applied Mathematics }

\begin{document}

	\begin{frontmatter}
	
		\cslyear{19}
		\cslreportnumber{1}
		
		\csltitle{Alternating Directions Implicit\\ Integration in a General Linear Method Framework}
		
		\cslauthor{Arash Sarshar, Steven Roberts, and Adrian Sandu}
		\cslciteas{\texttt{Arash Sarshar, Steven Roberts, Adrian Sandu,
				Alternating directions implicit integration in a general linear method framework,
				Journal of Computational and Applied Mathematics,
				2019,
				112619,
				\url{https://doi.org/10.1016/j.cam.2019.112619.}}}
		\csltitlepage

		\title{Alternating Directions Implicit Integration \\ in a General Linear Method Framework}
			
		%% Group authors per affiliation:
		\author[cs]{Arash Sarshar\corref{cor1}}
		\ead{sarshar@vt.edu}
	
		\author[cs]{Steven Roberts}
		\ead{steven94@vt.edu}
		
		\author[cs]{Adrian Sandu}
		\ead{sandu@cs.vt.edu}

		\cortext[cor1]{Corresponding author}
		\address[cs]{Computational Science Laboratory \\ Department of Computer Science \\ Virginia Tech}

\begin{abstract}
Alternating Directions Implicit (ADI) integration is an operator splitting approach to solve parabolic and elliptic partial differential equations in multiple dimensions based on solving sequentially a set of related one-dimensional equations. Classical ADI methods have order at most two, due to the splitting errors. Moreover, when the time discretization of stiff one-dimensional problems is based on Runge-Kutta schemes, additional order reduction may occur. This work proposes a new ADI approach based on the partitioned General Linear Methods framework. This approach allows the construction of high order ADI methods. Due to their high stage order, the proposed methods can alleviate the order reduction phenomenon seen with other schemes.  Numerical experiments are shown to provide further insight into the accuracy, stability, and applicability of these new methods. 
\end{abstract}
		
		\begin{keyword}
			Initial value problems, time integration, IMEX methods, alternating directions
			AMS 65L05 \sep AMS  65L07 
		\end{keyword}
		
	\end{frontmatter}
	%\linenumbers

%%%%%%%%%%%%%%%%%%%%%%%%%
\section{Introduction}
%%%%%%%%%%%%%%%%%%%%%%%%%

We are concerned with solving  the initial value problem: 
\begin{equation}
y'(t) = f(y) = \sum_{\sigma=1}^{N}f\comp{\sigma}(y),\quad y(t_0)=y_0, 
\label{eq:ode}
\end{equation}
where the right hand side function $f: \mathbb{R}^d \rightarrow \mathbb{R}^d$ is additively split into $N$ partitions. Systems such as \cref{eq:ode} emerge from method of lines semi-discretization of PDEs when all spatial derivatives are approximated by their discretization. In many cases, the right hand side function includes discrete self-adjoint operators performing spatial derivatives in different directions. The sparsity structure of these operators is similar, for example, in the case when a fixed-stencil finite difference method is used to resolve spatial derivatives. Implicit time-stepping methods are preferred to propagate stiff differential equations in time, but they require working with large Jacobian matrices. Implicit-Explicit (IMEX) methods allow us to integrate non-stiff parts of the system more efficiently, however, more can be achieved by designing specialized time-stepping methods for certain classes of problems. Depending on the choice of discretization, we can use the tensor product structure of derivative operators to only work with one-dimensional Jacobian matrices much smaller than the full Jacobian, applying directional derivatives in different directions in turn.

Alternating Directions Implicit (ADI) schemes for parabolic problems were first introduced in the works of Douglas \cite{Douglas_1955_ADI}, Douglas and Rachford \cite{Douglas_1956_ADI}, and Peaceman and Rachford \cite{Peaceman_1955_ADI}. Closely related to this field is the body of work on operator splitting schemes \cite{Strang_1968_splitting,Yoshida_1990_splitting,Yanenko_1971_book} and Approximate Matrix Factorizations (AMF) applied to Rosenbrock-W \cite{gonzalez2014rosenbrock,gonzalez2015amf,GONZALEZPINTO2018} and LIRK methods \cite{Sandu_2015_AMF-RK}. Another important development is the Fractional Step Runge--Kutta framework \cite{,Bujanda_2001_FSRK,Bujanda_2003_FSRK} investigating the link between directional methods and IMEX schemes. 

Early analysis of convergence of stiff ODEs can be traced back to Prothero-Robinson \cite{Prothero_1974_PR}. Ostermann \textit{et} al. formally show the fractional order phenomenon is related to changes in the behavior of local truncation error in stiff systems \cite{ostermann1992runge}. Methods of high stage order are known to alleviate this drawback \cite{BRAS201794,PORTERO2004409}. The General Linear Method (GLM) framework \cite{Jackiewicz_2009_book,Butcher_2001_GLM-stiff,Butcher_2003_construction} encompasses many of these methods and facilitates creation of new ones for novel applications. The theory of partitioned GLMs was formalized in \cite{Sandu_2014_IMEX-GLM} and different families of methods based on this structure have been reported in \cite{Sandu_2012_ICCS-IMEX,Sandu_2015_Stable_IMEX-GLM,Sandu_2016_highOrderIMEX-GLM,Sandu_2015_IMEX-TSRK}. More recent high order IMEX-GLMs found in the literature \cite{Izzo2019,SCHNEIDER2018121,schneider2019super} are based on Diagonally Implicit Multistage Integration Methods (DIMSIMs), Two-Step Runge--Kutta methods, and Peer methods providing various accuracy and stability enhancements. 

The goal of this paper is to extend the capabilities of ADI schemes to high order GLMs, creating methods resilient to order reduction while leveraging the efficiency of alternating implicit integration. 
The paper is organized as follows: We start by reviewing the partitioned GLM framework in \cref{sec:review-pglm}, introduce the structure of ADI-GLMs in \cref{sec:ADI-GLM-formulation}, study their order conditions in \cref{sec:ADI-GLM-OC}, and investigate their stability in \cref{sec:stability}. We comment on design principles and implementation aspects in \cref{sec:design} followed by numerical experiments in \cref{sec:experiments} and the concluding remarks in \cref{sec:conclusions}. \ref{app:methods} includes the coefficients of the new methods, and \ref{app:stability-plots} presents stability plots.

%%%%%%%%%%%%%%%%%%%%%%%%%%%%%%%
\section{ Traditional and partitioned General Linear Methods }
\label{sec:review-pglm}
%%%%%%%%%%%%%%%%%%%%%%%%%%%%%%%
A traditional GLM with $s$ internal and $r$ external stages represented by Butcher tableau:
\begin{align}
	\begin{array}{c|c|c}
		\c & \A & \U \\ \hline
   & \B & \V
	\end{array},
\end{align} 
advances the numerical solution to \cref{eq:ode} with timestep $h$ according to:
\begin{subequations}
	\begin{align}
			Y_i  &= h\sum_{j=1}^{s}a_{i,j} f(Y_j) + \sum_{j=1}^{r} u_{i,j}\,\xi_j^{[n-1]}, \quad
		i  = 1, \ldots,s, \\
		\xi_i^{[n]} &= h\sum_{j=1}^{s} b_{i,j} f(Y_j) + \sum_{j=1}^{r} v_{i,j} \xi_j^{[n-1]}, \quad
		i = 1,\ldots,r,
	\end{align}
\end{subequations}
where the matrix notation of coefficients is used:
\begin{align}
	\begin{aligned}
	& \A := [a_{i,j}] \in \mathbb{R}^{s\times s}, \quad \U := [u_{i,j}]\in \mathbb{R}^{s\times r}, \quad \B := [b_{i,j}]\in \mathbb{R} ^{r\times s}, \\    & \V := [v_{i,j}]\in \mathbb{R} ^{r\times r}, \quad
\W:=[w_{i,j}] = [\*w_0 \; \cdots \; \*w_p] \in \mathbb{R}^{r \times (p+1)},
	\end{aligned} 
\end{align}
where matrix $\W$ determines the relation between external stages and derivatives of the exact solution such that for a method of order $p$:
\begin{align*}
	\xi_i \tcomp{n} = \sum_{k=0}^{p} w_{i,k} h^k y^{(k)}(t_n) + \order{h^{p+1}}. 
\end{align*}
GLM framework is extensive and well-established. Readers interested in theoretical foundation of these methods are referred to the literature \cite{Jackiewicz_2009_book,Butcher_2001_GLM-stiff,Butcher_2003_construction}. 

IMEX-GLMs are extensions of traditional GLMs that treat partitions of the right hand side with different methods while keeping a single set of internal and external stages. One step of an IMEX-GLM formally reads as:
\begin{subequations}
	\label{eq:IMEX-GLM}
	\begin{align}
		Y_i  &= h\sum_{\sigma=1}^{N}\sum_{j=1}^{s}a\comp{\sigma}_{i,j} f\comp{\sigma}(Y_j) + \sum_{j=1}^{r} u_{i,j}\,\xi_j^{[n-1]}, \quad
		i  = 1, \ldots,s, \label{eq:IMEG-GLM-int-stage}\\
		\xi_i^{[n]} &= h \sum_{\sigma=1}^{N}\sum_{j=1}^{s} b_{i,j}\comp{\sigma} f\comp{\sigma}(Y_j) + \sum_{j=1}^{r} v_{i,j} \xi_j^{[n-1]}, \quad
		i = 1,\ldots,r. \label{eq:IMEG-GLM-ext-stage}
	\end{align}
\end{subequations}

%%%%%%%%%%%%%%%%%%%%%%%%%%%%%%%%%%%%%
\section{Formulation of ADI-GLMs}
\label{sec:ADI-GLM-formulation}
%%%%%%%%%%%%%%%%%%%%%%%%%%%%%%%%%%%%%

We rely on the theory of IMEX-GLMs as reported in  \cite{Sandu_2014_IMEX-GLM,Sandu_2015_Stable_IMEX-GLM,Sandu_2016_highOrderIMEX-GLM} to design partitioned GLMs suited for ADI integration. The goal is to construct GLMs that apply implicit integration to individual partitions of the right hand side function in \cref{eq:ode}, while using an explicit coupling to the other components. We seek to achieve high stage order while benefiting from the low computational cost of directional implicit methods. 

\begin{definition}[ADI-GLM schemes]
One step of an $N$-way partitioned ADI-GLM  applied to \cref{eq:ode} is defined as: 
\begin{subequations}
\label{eq:ADI-GLM-step}
\begin{align}
	Y_i\comp{\mu}  &= h\sum_{\sigma=1}^{N}\sum_{j=1}^{s}a\comp{\mu,\sigma}_{i,j} f\comp{\sigma}(Y_j\comp{\sigma}) + \sum_{\sigma=1}^{N}\sum_{j=1}^{r} u_{i,j}\comp{\mu, \sigma}\xi_j\super{\sigma}{[n-1]}, \label{ADI-GLM-step-int-stage} \\
	i & = 1, \ldots,s, \qquad \mu = 1,\ldots,N, \notag\\
	\xi_i\super{\mu}{[n]} &= h \sum_{\sigma=1}^{N}\sum_{j=1}^{s} b_{i,j}\comp{\mu,\sigma} f\comp{\sigma}(Y_j\comp{\sigma}) +  \sum_{\sigma=1}^{N}\sum_{j=1}^{r} v_{i,j}\comp{\mu,\sigma} \xi_j\super{\sigma}{[n-1]}, \label{ADI-GLM-step-ext-stage} \\
	i &= 1,\ldots,r, \qquad \mu = 1,\ldots,N. \notag 
\end{align}
\end{subequations}
\end{definition}
Here, we are interested in applying different combinations of explicit and diagonally implicit methods to the right hand side partitions and storing the resulting internal and external stages separately.
%As a consequence, we may have to carry separate external stages, but this comes at no extra cost as they are linear combinations of previously computed values.

If the method is order $p$, the external stages are related to derivatives of $y$ by:
\begin{align}
	\xi_i\super{\mu}{[n]} &=  w_{i,0}\comp{\mu} y(t_n) + \sum_{\sigma=1}^{N} \sum_{k=1}^{p} w_{i,k}\comp{\mu,\sigma} h^k (f\comp{\sigma})^{(k-1)}\left(y\left(t_n\right)\right) + \bigo{h^{p+1}}, \label{eq:ext-stages-nordsik} \\
\W\comp{\mu,\sigma}:& = [\*w_0\comp{\mu} \; \cdots \; \*w_p\comp{\mu,\sigma} ] \in \mathbb{R}^{r \times (p+1)}.
\end{align} 
The method is stage order $q$ if internal stages are approximations of the exact solution at  abscissa points $\c\comp{\mu}$:
\begin{equation}
	Y_i\comp{\mu} = y(t_{n-1} + \c_i\comp{\mu}h) + \bigo{h^{q+1}}.
	\label{eq:int-stages-order}
\end{equation}

\section{Construction of ADI-GLMs}
\label{sec:ADI-GLM-OC}
%%%%%%%%%%%%%%%%%%%%%%%%%%%%%%%%%%%%%
We start by considering a pair of explicit and implicit GLMs with the same number of external and internal stages:
\begin{equation}
	\label{eq:GLM-pair}
	\bgroup
	\def\arraystretch{1}
	\begin{butchertableau}{c|c|c}
	\c\comp{E}  & \A\comp{E}  & \U\comp{E}  \\ \hline 
	   		    & \B\comp{E}   & \V\comp{E}
	\end{butchertableau}\,,
	\egroup \quad
		\bgroup
	\def\arraystretch{1}
	\begin{butchertableau}{c|c|c}
	\c\comp{I}  & \A\comp{I}  & \U\comp{I}  \\ \hline 
	& \B\comp{I}   & \V\comp{I}
	\end{butchertableau}.
	\egroup
\end{equation}
We construct  ADI-GLMs using a collection of IMEX-GLMs each  performing implicit integration in a specific direction. A preconsistent IMEX-GLM has order $p$ and stage order $q \in \{p, p-1\}$ if and only if the following conditions hold: 
\begin{subequations}
	\begin{align}
		\frac{  \c\super{\sigma}{\times k} }{k!} - \frac{\A \c\super{\sigma}{\times (k-1)}}{(k-1)!} - \U\comp{\sigma} \*w\comp{\sigma}_k &= 0, \\
		 \quad k = \{1,\ldots,q\}, &\quad \sigma \in \{E,I\}, \notag \\
		\sum_{l=0}^{k} \frac{\*w\comp{\sigma}_{k-l}}{l!} - \frac{\B\comp{\sigma} \c\super{\sigma}{\times (k-1)}}{(k-1)!} - \V\comp{\sigma} \*w\comp{\sigma}_k &= 0,\\
		 \quad k = \{1,\ldots,p\}, & \quad \sigma \in \{E,I\}. \notag
	\end{align}
	\label{eq:GLM-OC}
\end{subequations}
The structure of the Butcher tableau for an ADI-GLM  depends on the number of partitions and number of stiff partitions that require implicit treatment. Here, we focus on three practical examples and more elaborate designs follow the same principles. The Butcher tableau for a 3-way partitioned ADI-GLM with alternating implicit stages in all partitions is:
\begin{equation}
	\begin{butchertableau}{c|c c c|c c c}
		\c & \A\comp{I} & \A\comp{E} & \A\comp{E} & \U & \0 & \0 \\
		\c & \A\comp{I} & \A\comp{I} & \A\comp{E} & \0 & \U  & \0 \\ 
		\c & \A\comp{I} & \A\comp{I} & \A\comp{I} & \0 & \0  & \U\\ \hline 
		& \B\comp{I} & \B\comp{E} & \B\comp{E} & \V & \0 & \0  \\ 
		& \B\comp{I} & \B\comp{I} & \B\comp{E} & \0 & \V & \0  \\ 
		& \B\comp{I} & \B\comp{I} & \B\comp{I} & \0 & \0 & \V   
	\end{butchertableau}\,.
	\label{eq:ADI-DIMSIM-3D}
\end{equation}

When only two partitions are stiff, the non-stiff partition is carried through explicitly:
\begin{equation}
\begin{butchertableau}{c|c c c|c c c}
\c & \A\comp{I} & \A\comp{E} & \A\comp{E} & \U & \0  & 0 \\
\c & \A\comp{I} & \A\comp{I} & \A\comp{E} & \0 & \U   & 0 \\ 
\c & \A\comp{I} & \A\comp{I} & \A\comp{E} & \0 & 0  & \U \\ \hline 
& \B\comp{I} & \B\comp{E} & \B\comp{E} &   \V & \0 & 0   \\ 
& \B\comp{I} & \B\comp{I} & \B\comp{E} & \0 & \V& 0  \\
 & \B\comp{I} & \B\comp{I} & \B\comp{E} & \0 & 0 & \V  \\   
\end{butchertableau}\,.
\label{eq:ADI-DIMSIM-2D-3Part}
\end{equation}
We notice immediately that $Y\comp{3}_i \equiv Y\comp{2}_i$, therefore one only computes two types of stage vectors, and the second is used as an argument for the explicit integration of the third, non-stiff component.
%\textcolor{red}{I thought about this one for a while. I think we can have a non-square A matrix, (as we do in IMEX-GLM) but we have to explain what the argument of each function evaluation is. In this case we only have stages $Y\comp{1}_i$ and $Y\comp{2}_i$ although we have 3 partitions.  I decided to expand the table and explain why we don't actually compute $Y\comp{3}_i $in a remark to keep the uniformity of the paper. Hopefully this makes sense.}	

In a similar fashion, a 2-way partitioned ADI-GLM is described by:
\begin{equation}
	\begin{butchertableau}{c|c c |c c }
	\c & \A\comp{I} & \A\comp{E} & \U & \0   \\
	\c & \A\comp{I} & \A\comp{I} & \0 & \U   \\  \hline 
	    & \B\comp{I} & \B\comp{E} &  \V & \0   \\ 
	    & \B\comp{I} & \B\comp{I}  &   \0 & \V   
	\end{butchertableau}\,.
	\label{eq:ADI-DIMSIM-2D}
\end{equation}
\begin{remark}
	Comparing \cref{eq:ADI-DIMSIM-3D,eq:ADI-DIMSIM-2D-3Part,eq:ADI-DIMSIM-2D} with \cref{eq:ADI-GLM-step}, notice that we have chosen:
	\begin{subequations}
		\begin{align}
			\c\comp{E} &= \c\comp{I} = \c, \\
			\U\comp{E} &= \U\comp{I} = \U, \\
			\V\comp{E} &= \V\comp{I} = \V.
		\end{align}
		\label{eq:U-V-consistency}
	\end{subequations}
	 This selection is practically useful in creating IMEX-GLMs with unified internal stages. In the context of ADI-GLMs this choice allows us to keep the number of internal and external stages as low as the number of stiff partitions.
	\end{remark}
	
\begin{remark}
 We have also decoupled  computations involving the external stages: 
	\begin{align}
		\label{eq:U-V-diagonality}
	   \U\comp{\sigma,\mu} = 
	   		\begin{cases}
			    \0 ~~  & \sigma \neq \mu \\
			    \U ~~ & \sigma = \mu 
			\end{cases}, \qquad 
   		\V\comp{\sigma,\mu} = \begin{cases}
	     \0 ~~  & \sigma \neq \mu \\
	     \V ~~ & \sigma = \mu 
	     \end{cases}. \qquad 
	\end{align}
\end{remark}

\begin{theorem}
	\label{thm:ADI-GLM-oc-thrm}
	The ADI-GLMs \cref{eq:ADI-GLM-step} subject to \cref{eq:U-V-consistency,eq:U-V-diagonality} is stage order $q$ and order $p$, hereafter denoted by order $(q,p)$, if and only if individual methods \eqref{eq:GLM-pair} are order $(q,p)$.
\end{theorem}

\begin{proof}
	We first assume that the ADI-GLM is order $(q,p)$ such that \cref{eq:ext-stages-nordsik,eq:int-stages-order} hold. Since all internal stages $Y\comp{\sigma}_i$ share the same abscissa, from \cref{eq:int-stages-order} we have:
	\begin{equation}
		Y\comp{\sigma}_i = Y\comp{\mu}_i + \bigo{h^{q+1}}, \qquad \sigma,\mu \in \{1,\ldots, N \}. 
		\label{eq:stage-order-equivalency}
	\end{equation} 
	Therefore, we can replace  $Y\comp{\sigma}_j$ with $Y\comp{\mu}_j$ in \cref{ADI-GLM-step-int-stage} without changing the order. The resulting method is an IMEX-GLM with 
	\begin{equation}
	\begin{butchertableau}{c|c c c|c  }
	\c & \A\comp{\mu,1} & \cdots        & \A\comp{\mu,N} & \U  \\  \hline 
	    & \B\comp{\mu,1}  & \cdots       & \B\comp{\mu,N}  & \V   
	\end{butchertableau}, \qquad \mu \in \{1, \ldots, N\}.
	\label{eq:ADI-to-IMEX}
	\end{equation}
	From IMEX-GLM order conditions \cite{Sandu_2014_IMEX-GLM,zhang2016high}  method \eqref{eq:ADI-to-IMEX} is order $(q,p)$ if and only if individual methods 
\begin{equation*}
		\begin{butchertableau}{c|c|c  }
		\c & \A\comp{\mu,\sigma} & \U  \\  \hline 
		   & \B\comp{\mu,\sigma}  & \V   
		\end{butchertableau},   \qquad \mu \in \{1, \ldots, N\}, \qquad \sigma \in \{1, \ldots, N\}.
\end{equation*}
 are order $(q,p)$. This means that the methods in \cref{eq:GLM-pair} have to be order $(q,p)$.
	
The \textit{if} part of the theorem can be proven along the same line of reasoning. Assuming individual methods  \eqref{eq:GLM-pair} are order $(q,p)$ the IMEX-GLM \eqref{eq:ADI-to-IMEX} is order $(q,p)$. Internal stage values in \cref{eq:IMEG-GLM-int-stage} can be replaced by an approximation of the same order as in \cref{eq:stage-order-equivalency} to create the internal stages for ADI-GLM. Since the order of internal stages has not changed, external stages also remain order $p$. This concludes the proof.
\end{proof}
\begin{remark}
	A corollary to \cref{thm:ADI-GLM-oc-thrm} is that in the case of ADI-GLM \eqref{eq:ADI-DIMSIM-2D-3Part}, we can forgo computing $\mleft(  Y\comp{3}_i, \xi\super{3}{[n]} \mright)$ stages without losing accuracy. Furthermore, this choice will not affect the stability since the stiff partitions are still treated implicitly and the integration of the non-stiff partition already appears in stage computations.
\end{remark}
\begin{comment}
To construct a method suitable for  and ADI split problem, first we have to make sure that \cref{eq:PGLM-order-condition} holds for all partitions. For a two dimensional problem and a pair of Implicit and explicit DIMSIMS \cref{eq:2-way-partitioned-PGLM-butcher-tab} looks like: 
%
\begin{equation}
	\label{eq:2D-ADI-PGLM-butcher-tab}
	\bgroup
	\def\arraystretch{1.75}
	\begin{array}{c|c c|c c}
		\c\comp{1} & \A\comp{I} & \A\comp{E} & \U\comp{I} & \U\comp{E}  \\
		\c\comp{2} & \A\comp{I} & \A\comp{I} & \U\comp{I} & \U\comp{I}  \\ \hline 
		& \B\comp{I} & \B\comp{E} & \V\comp{I} & \V\comp{E}  \\ 
		& \B\comp{I} & \B\comp{I} & \V\comp{I} & \V\comp{I}  \\ 
	\end{array}\,.
	\egroup
\end{equation}
%
and \cref{eq:IMEX-simplifications} requires that 
\begin{subequations}
\begin{align}
	\U\comp{I} &= \U\comp{E} = \U \,, \\
	\V\comp{I} &= \V\comp{E} = \V \,, 
\end{align}
\end{subequations}
and the combined tableau is : 
%
\begin{equation}
	\label{eq:IMEX-DIMSIM-ADI-butcher-tab}
	\bgroup
	\def\arraystretch{1.75}
		\begin{array}{c|c c|c c}
		\c\comp{1} & \A\comp{I} & \A\comp{E} & \U \\
		\c\comp{2} & \A\comp{I} & \A\comp{I} & \U \\ \hline 
		& 2 \B\comp{I} &   \B\comp{I} + \B\comp{E}   & 2 \V
	\end{array}\,.
	\egroup
\end{equation}
%
\end{comment}

\section{Stability of ADI-GLMs}
\label{sec:stability}
%%%%%%%%%%%%%%%%%%%%%%%%
Applying the ADI-GLM \eqref{eq:ADI-DIMSIM-3D} to the linear scalar test equation:
\begin{equation}
	u' = \lambda_x u + \lambda_y u + \lambda_z u,
	\label{eq:dalquitst-test}
\end{equation}
and using \cref{eq:ADI-GLM-step} leads to the following directional stages:
\begin{subequations}
	\begin{align}
	Y\comp{1} &= \eta_x \A\comp{I} Y\comp{1} + \eta_y \A\comp{E} Y\comp{2} + \eta_z \A\comp{E} Y\comp{3} + \U \xi\super{1}{[n-1]},\\
	Y\comp{2} &= \eta_x \A\comp{I} Y\comp{1} + \eta_y \A\comp{I} Y\comp{2} + \eta_z \A\comp{E} Y\comp{3} + \U \xi\super{2}{[n-1]},\\		
	Y\comp{3} &= \eta_x \A\comp{I} Y\comp{1} + \eta_y \A\comp{I} Y\comp{2} + \eta_z \A\comp{I} Y\comp{3} + \U \xi\super{3}{[n-1]},\\
	\xi\super{1}{[n]} &= \eta_x \B\comp{I} Y\comp{1} + \eta_y \B\comp{E} Y\comp{2} + \eta_z \B\comp{E} Y\comp{3} + \V \xi\super{1}{[n-1]},\\
	\xi\super{2}{[n]} &= \eta_x \B\comp{I} Y\comp{1} + \eta_y \B\comp{E} Y\comp{2} + \eta_z \B\comp{E} Y\comp{3} + \V \xi\super{2}{[n-1]},\\
	\xi\super{3}{[n]} &= \eta_x \B\comp{I} Y\comp{1} + \eta_y \B\comp{I} Y\comp{2} + \eta_z \B\comp{I} Y\comp{3} + \V \xi\super{3}{[n-1]},
	\end{align}
\end{subequations}
where  $\eta_x = h\lambda_x, \eta_y = h \lambda_y, \eta_z = h \lambda_z$.
Defining auxiliary notations $\*Z = \blkdiag{\mleft( \eta_x \*I_{s \times s}, \eta_y \*I_{s \times s}, \eta_z \*I_{s \times s} \mright)}$ and $\xi^{[n]} = \mleft(\xi\super{1}{[n]}, \xi\super{2}{[n]}, \xi\super{3}{[n]} \mright)^T $,
the stability matrix is defined as:  
\begin{subequations}
	\begin{align}
	\xi^{[n]} &= \mathbf{M}( \eta_x, \eta_y, \eta_z )\,\xi^{[n-1]}, \\
	\mathbf{M}( \eta_x, \eta_y, \eta_z ) &=   \~{\*V} +  \~{\*B} \, \*Z \mleft( \mathbf{I}_{3s \times 3s} - \~{\*A} \*Z \mright )^{-1} \~{\*U}, 
	\end{align}
\end{subequations}
where:
\begin{subequations}
	\begin{align}
	\~{\*A} &= 
	\begin{bmatrix}
	\A\comp{I} & \A\comp{E} & \A\comp{E}\\
	\A\comp{I} & \A\comp{I} &  \A\comp{E} \\
	\A\comp{I} & \A\comp{I} &  \A\comp{I} 
	\end{bmatrix},\qquad 
	\~{\*U} = \*I_{3 \times 3} \otimes \U,
	\\
	\~{\*B} &= 
	\begin{bmatrix}
	\B\comp{I} &  \B\comp{E} & \B\comp{E} \\
	\B\comp{I} &  \B\comp{I} &  \B\comp{E} \\
	\B\comp{I} &  \B\comp{I} &  \B\comp{I} 
	\end{bmatrix} ,\qquad 
	\~{\*V} = \*I_{3 \times 3} \otimes \V .
	\end{align}
\end{subequations}
When the eigenvalues of the system \eqref{eq:dalquitst-test} are equal in all directions  such that $\eta_x = \eta_y = \eta_z = \eta$ the stability matrix becomes:
\begin{equation}
	\^{\*M}(\eta) = \*M (\eta, \eta, \eta) = \~{\*V} + {\eta}\~{\*B} \mleft(  \*I_{3s\times 3s} - {\eta} \~{\*A} \mright)^{-1} \~{\*U}.
	\label{eq:stab-equal-eigs}
\end{equation}
\Cref{eq:stab-equal-eigs} provides practical means for assessment and optimization of stability of ADI-GLMs. 
\begin{remark}
	The stability regions for individual explicit and implicit methods are defined as:
	\begin{subequations}
	\begin{align}
			\mathcal{S}\comp{\sigma} &= \mleft\{ \eta \in \mathbb{C} : \; \*M\comp{\sigma}(\eta) \text{ power bounded} \mright\},\\
		    \*M\comp{\sigma}(\eta) &= \V\comp{\sigma} + \eta \B\comp{\sigma} \mleft(  \*I_{s\times s} - \eta \A\comp{\sigma}\mright)^{-1}\U\comp{\sigma} ,\quad \sigma \in \{E,I \}.
	\end{align}
	\end{subequations}
	The stability region of a 3-way partition method is defined as: 
	\begin{align}
		S &= \mleft \{ \eta_x, \eta_y, \eta_z \in \mathbb{C}  : \; \*M(\eta_x, \eta_y, \eta_z) \text{ power bounded} \mright \}.
	\end{align}
\end{remark}
\begin{remark}
	To investigate the stability of ADI-GLMs we define real and complex stability regions as:
	\begin{subequations}
		\begin{align}
			\mathcal{S}_{\mathrm{Real}} & = \mleft \{ \eta_x,\eta_y  \in \mathbb{R} : \; \*M \mleft(\eta_x, \eta_y, \max \mleft(\eta_x,\eta_y  \mright) \mright)  \text{ power bounded}  \mright \},\\
			\mathcal{S}_{\mathrm{Cplx}} & = \mleft\{ \eta \in \mathbb{C}  : \; \^{\*M}(\eta) \text{ power bounded} \mright\}.
		\end{align}
	\end{subequations}
\end{remark}

\begin{remark}[Stability as all partitions become infinitely stiff]
Consider the stability matrix \cref{eq:stab-equal-eigs} when the eigenvalues in each direction simultaneously approach $-\infty$:
\begin{equation} \label{eq:mhat_limit}
	\lim_{\eta \to - \infty} \^{\*M}(\eta) = \begin{bmatrix}
		\*V\comp{I} - \left( \*B\comp{I} - \*B\comp{E} \right) \left( \*A\comp{I} - \*A\comp{E} \right)^{-1} \*U & \boldsymbol{*} \\
		\0 & \*M\comp{I}(-\infty)
	\end{bmatrix}.
\end{equation}
Due to the block triangular structure of this matrix, the eigenvalues of \cref{eq:mhat_limit} are the eigenvalues of the diagonal blocks and the entries in the upper right block can be ignored.

Consider the case $p = q = r = s$.  We will further assume $\*w\comp{I}_0 = \*w\comp{E}_0$, which comes at no loss of generality since we can always pick an equivalent formulation of the base methods where this holds. Using the difference of the order conditions of the base methods, we have that
\begin{subequations}
	\begin{align}
		\left( \A\comp{I} - \A\comp{E} \right) \*C + \U \left( \W_{:, 1:p}\comp{I} - \W_{:, 1:p}\comp{E} \right) &= \0, \\
		\left( \W_{:, 1:p}\comp{I} - \W_{:, 1:p}\comp{E} \right) \boldsymbol{\mu} - \left( \B\comp{I} - \B\comp{E} \right) \*C - \V \left( \W_{:, 1:p}\comp{I} - \W_{:, 1:p}\comp{E} \right) &= \0,
	\end{align}
\end{subequations}
where
\begin{equation}
	\mu_{i,j} = \begin{cases}
		0 & i > j \\
		\frac{1}{(j - i)!} & i \le j
	\end{cases},
	\qquad
	\*C = \begin{bmatrix}
		\one_s & \c & \frac{\c^2}{2} & \cdots & \frac{\c^{p-1}}{(p-1)!}
	\end{bmatrix}.
\end{equation}

Now we have that
\begin{equation}
	\begin{split}
		& \quad \left( \*V\comp{I} - \left( \*B\comp{I} - \*B\comp{E} \right) \left( \*A\comp{I} - \*A\comp{E} \right)^{-1} \*U \right) \left( \*W_{:, 1:p}\comp{I} - \*W_{:, 1:p}\comp{E} \right) \\
		&= \*V\comp{I} \left( \*W_{:, 1:p}\comp{I} - \*W_{:, 1:p}\comp{E} \right) + \left( \*B\comp{I} - \*B\comp{E} \right) \*C \\
		&= \left( \*W_{:, 1:p}\comp{I} - \*W_{:, 1:p}\comp{E} \right)  \boldsymbol{\mu}.
	\end{split}
\end{equation}
Thus, the upper left block of \cref{eq:mhat_limit} is similar to $\boldsymbol{\mu}$ provided $\W_{:, 1:p}\comp{I} - \W_{:, 1:p}\comp{E}$ is non-singular.  In this case, \cref{eq:mhat_limit} is not power bounded because the $1$ eigenvalue of $\boldsymbol{\mu}$ is defective.  We note that this is not an issue when only a single eigenvalue becomes infinitely stiff.
\end{remark}

In \ref{app:stability-plots} we provide plots of different stability regions for ADI-GLMs. 
\section{Design and implementation of  ADI-GLMs}
\label{sec:design}
%%%%%%%%%%%%%%%%%%%%%%%%
We have chosen the GLMs to be DIMSIMs \cite{Butcher_1993_diagonal} in order to reduce the number of free parameters in the design and simplify the order conditions. We require:
\begin{subequations}
	\begin{align}
	a_{i,i}\comp{I} &= \gamma, \quad a_{i,i}\comp{E} = 0, \quad a_{i,j}\comp{\sigma} =0, \quad \text{for} \quad j>i,  &\sigma \in \{E,I \}, \\
	\U\comp{\sigma} &= \*I_{s\times r},  &\sigma \in \{E,I \},\\
	\V\comp{\sigma} &= \one^T_r v, \quad v^T \one_r = 1, &\sigma \in \{E,I \}.
	\end{align}
\end{subequations}
\method{}s derived in this paper have $p=q=r=s$.
 The design process starts with choosing the abscissa vector $\c$. The remaining free parameters are coefficients of $\A\comp{E}$, $\A\comp{I}$, and $v$. For the new second and third order schemes, we picked existing, L-stable, type 2 DIMSIMs for $\A\comp{I}$ and $v$.  Then, we choose $\A\comp{E}$ by numerically optimizing the area of the $\mathcal{S}_{\mathrm{Cplx}}$ and $\mathcal{S}\comp{E}$ stability regions using Mathematica.  At fourth order, we performed the same optimization for $\A\comp{E}$, however, we were unable to achieve satisfactory stability when using an existing type 2 DIMSIM for the implicit base method.  Instead, we derived a new A$(\ang{83})$-stable DIMSIM for which the ADI-GLM stability was acceptable.

Once $\A\comp{E}$ and $\A\comp{I}$ and $v$ are determined, $\B\comp{E}$ and $\B\comp{I}$ are given using DIMSIM formulas \cite{jackiewicz2009general,butcher1993diagonally}. $\W\comp{I}$ and $\W\comp{E}$ are computed by solving \cref{eq:GLM-OC} and used in the starting procedure  to generate initial values of the external stages at the beginning of the time-stepping loop in \cref{eq:ADI-GLM-step}. The starting procedure consists of integrating the system  \cref{eq:ode} exactly over a short time-span $\left[0,(p-1) H \right]$ and using function values \begin{equation}
	f_k\comp{\sigma}:=f\comp{\sigma}\mleft( y\mleft( k H\mright) \mright), \quad  k = \{0, \ldots, p-1\}, \quad \sigma \in\{ 1,\ldots,N\},
\end{equation}
to approximate, via finite differences, the higher order derivatives needed in \cref{eq:ext-stages-nordsik}. Readers interested in further details about the starting procedure may consult \cite{califano2017starting,Sandu_2014_IMEX-GLM}. The ending procedure for GLMs produces the high order approximation to $y(t_f)$ at the final time using stage values. All ADI-GLMs designed in this paper have the property that $\c_{s} = 1$, therefore, the last computed internal stage may be used as the final value in the integration with no further calculation required:
	\begin{equation}
	y_{t_f} = Y_s\comp{N}  = h\sum_{\sigma=1}^{N}\sum_{j=1}^{s}a\comp{N ,\sigma}_{s,j} f\comp{\sigma}(Y_j\comp{\sigma}) + \sum_{j=1}^{r} u_{s,j}\comp{N , \sigma}\xi_j\super{\sigma}{[n-1]}.
	\end{equation}

\begin{remark}[The ADI character of the methods]
	The Butcher tableau for ADI-GLMs can be permuted to reflect the order of computation of stages in practice. In general, an ADI-GLM proceeds with computing internal stages: 
	\begin{equation}
	\mleft\{ Y\comp{1}_1, Y\comp{2}_1, \ldots, Y\comp{N}_1, \ldots, Y\comp{1}_s, \ldots, Y\comp{2}_s, \ldots, Y\comp{N}_s \mright\},
	\end{equation} 
	after which external stage updates are computed. Let us consider the application of the second order ADI-GLM \cref{eq:ADI-DIMSIM2-GARK-table} to \cref{eq:ADI-DIMSIM-2D}. We reorder the tableau according to the permutation list $\mathcal{P} = \{1,3,2,4\}$ to get the permuted tableau \cref{eq:ADI-DIMSIM2-GARK-table-permuted}. 
	\begin{subequations}
	\begin{align}
	\begin{butchertableau}{c|c|c}
	\c & \A & \U \\ \hline
	& \B & \V
	\end{butchertableau} & =
	\def\arraystretch{0.8}
	\begin{butchertableau}[1.3]{c|cccc|cccc}
	0 & \frac{5}{8} & 0 & 0 & 0 & 1 & 0 & 0 & 0 \\
	1 & \frac{1}{4} & \frac{5}{8} & \frac{1}{2} & 0 & 0 & 1 & 0 & 0 \\
	0 & \frac{5}{8} & 0 & \frac{5}{8} & 0 & 0 & 0 & 1 & 0 \\
	1 & \frac{1}{4} & \frac{5}{8} & \frac{1}{4} & \frac{5}{8} & 0 & 0 & 0 & 1 \\ \hline
	& \frac{1}{2} & -\frac{5}{32} & -\frac{3}{128} & \frac{5}{128} & -\frac{5}{16} & \frac{21}{16} & 0 & 0 \\
	& 0 & \frac{27}{32} & \frac{13}{128} & \frac{85}{128} & -\frac{5}{16} & \frac{21}{16} & 0 & 0 \\
	& -\frac{3}{128} & \frac{5}{128} & -\frac{3}{128} & \frac{5}{128} & 0 & 0 & -\frac{5}{16} & \frac{21}{16} \\
	& \frac{13}{128} & \frac{85}{128} & \frac{13}{128} & \frac{85}{128} & 0 & 0 & -\frac{5}{16} & \frac{21}{16} \\
	\end{butchertableau}, \label{eq:ADI-DIMSIM2-GARK-table} \\
	\begin{butchertableau}{c|c|c}
	\c & \A_{\mathcal{P},\mathcal{P}} & \U_{\mathcal{P},:} \\ \hline
	& \B_{:,\mathcal{P}} & \V
	\end{butchertableau} & =
	\def\arraystretch{0.8}
	\begin{butchertableau}[1.3]{c|cccc|cccc}
	0 & \frac{5}{8} & 0 & 0 & 0 & 1 & 0 & 0 & 0 \\
	0 & \frac{5}{8} & \frac{5}{8} & 0 & 0 & 0 & 0 & 1 & 0 \\
	1 & \frac{1}{4} & \frac{1}{2} & \frac{5}{8} & 0 & 0 & 1 & 0 & 0 \\
	1 & \frac{1}{4} & \frac{1}{4} & \frac{5}{8} & \frac{5}{8} & 0 & 0 & 0 & 1 \\ \hline
	& \frac{1}{2} & -\frac{3}{128} & -\frac{5}{32} & \frac{5}{128} & -\frac{5}{16} & \frac{21}{16} & 0 & 0 \\
	& 0 & \frac{13}{128} & \frac{27}{32} & \frac{85}{128} & -\frac{5}{16} & \frac{21}{16} & 0 & 0 \\
	& -\frac{3}{128} & -\frac{3}{128} & \frac{5}{128} & \frac{5}{128} & 0 & 0 & -\frac{5}{16} & \frac{21}{16} \\
	& \frac{13}{128} & \frac{13}{128} & \frac{85}{128} & \frac{85}{128} & 0 & 0 & -\frac{5}{16} & \frac{21}{16} \\
	\end{butchertableau}. \label{eq:ADI-DIMSIM2-GARK-table-permuted}
	\end{align}
	\end{subequations}	
	We observe how the lower triangular structure of $\A_{\mathcal{P},\mathcal{P}}$ defines successive implicit stages in different directions while using previously computed stage values explicitly. 
\end{remark}

\section{Numerical Experiments}
\label{sec:experiments}

In this section, we investigate numerically the accuracy and stability of ADI-DIMSIMs using 2D and 3D time-dependent parabolic PDEs. Up to this point, we have only considered autonomous problems, however, ADI-GLM extends to non-autonomous systems by evaluating the right hand side functions at the consistent times $t_{n-1} + \c h$.  For a 3D problem we use the equation: 

\begin{subequations}
\begin{align}
	\frac{\partial u}{\partial t} &= \frac{\partial^2 u}{\partial x^2} + \frac{\partial^2 u}{\partial y^2} + \frac{\partial^2 u}{\partial z^2}+ g(x,y,z,t),\\
	%\frac{\partial{u}}{\partial{n}} &= 0,\\
	\begin{split}
	g(x,y,z,t) &= e^t (1-x) x (1-y) y (1-z) z \\
	&\qquad +2 e^t (1-x) x (1-y) y +2 e^t (1-x) x (1-z) z \\
	&\qquad +2 e^t (1-y) y (1-z)z-6 e^t\\
	& \qquad +e^t \left(\left(x+\frac{1}{3}\right)^2+\left(y+\frac{1}{4}\right)^2+\left(z+\frac{1}{2
	}\right)^2\right),
	\end{split}
\end{align}
\label{eq:3D-heat-problem}
\end{subequations}
with Dirichlet boundary conditions according to the exact solution: 
\begin{align}
	u(x,y,z,t) &=e^t (1-x) x (1-y) y (1-z) z \\
	& \qquad +  e^t \mleft(\mleft(x+\frac{1}{3}\mright)^2+\mleft(y+\frac{1}{4}\mright)^2+\mleft(z+\frac{1}{2}\mright)^2\mright).
\end{align}	
The spatial discretization uses second order finite differences on the unit cube domain $D:=\{ x,y,z \in [0 ,1] \}$ with a uniform mesh with $N_p$ points in each direction. We use the parameter $N_p$ in our experiments to change the stiffness of directional derivatives.   Note that using a uniform mesh allows us to factorize a tridiagonal 1D Jacobian matrix once and use it to efficiently to compute directional stages.

To verify the temporal order of convergence for the new methods, we integrate the problem over a time-span $t = [0,1]$ and record the relative $\ell_2$ error at final time versus number of time steps. \Cref{fig:3D-conv-DIMSIM2,fig:3D-conv-DIMSIM3,fig:3D-conv-DIMSIM4}, verify the theoretical order for a range of mesh sizes. We compare \method{}s with an ADI scheme based on a fourth order IMEX Runge--Kutta method reported in \cite[Example 3]{Sandu_2015_GARK}. We note the deterioration in the order as the problem becomes more stiff with decreasing mesh size in \cref{fig:3D-conv-IMEX4}.
%%%%%%%%%%%%%%3D-Conv-Plots%%%%%%%%%%%%%%%%
\begin{figure}[tbhp]
	\centering
	\ref{mylegend}
	\begin{subfigure}[t]{0.45\textwidth}
			\begin{tikzpicture}
					\begin{loglogaxis}[height=1.8in, grid=major, xlabel={Steps}, ylabel={Error},
					 xmin=80, xmax=1200,
					legend entries={$N_p=4\quad$,  $N_p=16\quad$,  $N_p=256\quad$,
						$N_p=32\quad$, $N_p=64\quad$, $N_p=128\quad$ }, legend style={at={(0.5, 1.03)},anchor=south}, legend columns=4, legend to name=mylegend]
					\addplot table [x index=0, y index=1] {Data_v1/3D/ADI_DIMSIM2_3D_S_result.txt}
					coordinate [pos=0.25] (A)
					coordinate [pos=0.50] (B);
					\draw (A) |- (B)
					node [anchor = north, pos=0.70] {$m=2$};
					\addplot table [x index=0, y index=2] {Data_v1/3D/ADI_DIMSIM2_3D_S_result.txt};
					\addplot table [x index=0, y index=3] {Data_v1/3D/ADI_DIMSIM2_3D_S_result.txt};
			\end{loglogaxis}
			\end{tikzpicture}
			\caption{\method{2}}
			\label{fig:3D-conv-DIMSIM2}
	\end{subfigure}
	\begin{subfigure}[t]{0.45\textwidth}
		\begin{tikzpicture}
				\begin{loglogaxis}[height=1.8in, grid=major, xlabel={Steps}, ylabel={Error},
				xmin=80, xmax=1200,ymin=10E-11]
				\addplot table [x index=0, y index=1] {Data_v1/3D/ADI_DIMSIM3_3D_S_result.txt}
				coordinate [pos=0.25] (A)
				coordinate [pos=0.50] (B);
				\draw (A) |- (B)
				node [anchor = north, pos=0.7] {$m=3$};
				\addplot table [x index=0, y index=2] {Data_v1/3D/ADI_DIMSIM3_3D_S_result.txt};
				\addplot table [x index=0, y index=3] {Data_v1/3D/ADI_DIMSIM3_3D_S_result.txt};
				\end{loglogaxis}
				\end{tikzpicture}
		\caption{\method{3}}
		\label{fig:3D-conv-DIMSIM3}
	\end{subfigure} 
	\\
	\begin{subfigure}[t]{0.45\textwidth}
		\begin{tikzpicture}
					\begin{loglogaxis}[height=1.8in, grid=major, xlabel={Steps}, ylabel={Error},
					xmin=80, xmax=1200, ymax=1]
					\addplot table [x index=0, y index=1] {Data_v1/3D/ADI_DIMSIM4_3D_S_result.txt}
					coordinate [pos=0.25] (A)
					coordinate [pos=0.50] (B);
					\draw (A) |- (B)
					node [anchor = north, pos=0.75] {$m=4$};
					\addplot table [x index=0, y index=2] {Data_v1/3D/ADI_DIMSIM4_3D_S_result.txt};
					\addplot table [x index=0, y index=3] {Data_v1/3D/ADI_DIMSIM4_3D_S_result.txt};
					\end{loglogaxis}
		\end{tikzpicture}
		\caption{\method{4}}
		\label{fig:3D-conv-DIMSIM4}
	\end{subfigure} 
	\begin{subfigure}[t]{0.45\textwidth}
		\begin{tikzpicture}
				\begin{loglogaxis}[height=1.8in, grid=major, xlabel={Steps}, ylabel={Error},
				xmin=80, xmax=1200]
				\addplot table [x index=0, y index=1] {Data_v1/3D/ADI_RK4_3D_result.txt}
				coordinate [pos=0.25] (A)
				coordinate [pos=0.50] (B);
				\draw (A) |- (B)
				node [anchor = north, pos=0.75] {$m=4$};
				\addplot table [x index=0, y index=2] {Data_v1/3D/ADI_RK4_3D_result.txt}
				coordinate [pos=0.50] (E)
				coordinate [pos=0.70] (F);
				\draw (E) |- (F)
				node [anchor = north, pos=0.75] {$m=3$};
				\addplot table [x index=0, y index=3] {Data_v1/3D/ADI_RK4_3D_result.txt}
				coordinate [pos=0.65] (C)
				coordinate [pos=1.00] (D);
				\draw (C) |- (D)
				node [anchor = north, pos=0.8] {$m=2$};
				\end{loglogaxis}
		\end{tikzpicture}
		\caption{IMEX-RK4}
		\label{fig:3D-conv-IMEX4}
	\end{subfigure} 
	\caption{Convergence plots for \method{}s on 3D test problem compared to IMEX-RK4 method}
	\label{fig:ADI_3D_conv}
\end{figure}
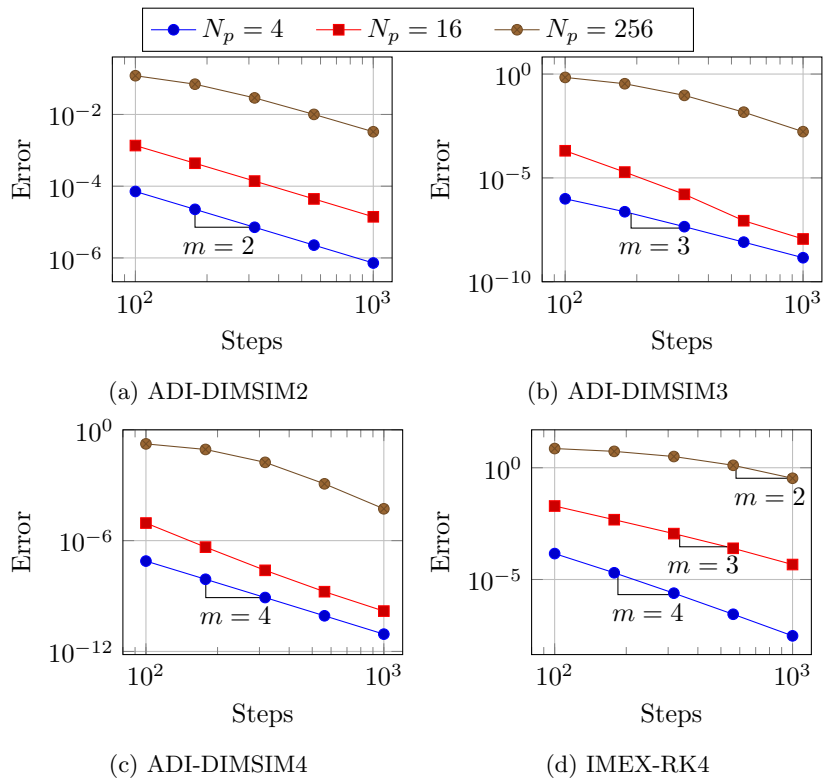

For a 2D numerical experiment the following problem is used on unit square domain $D= \{ x,y \in [0,1] \}$, with the same spatial discretization and integrated over the same time-span:
 \begin{subequations}
 	\begin{align}
 	\frac{\partial u}{\partial t} &= \frac{\partial^2 u}{\partial x^2} + \frac{\partial^2 u}{\partial y^2} + h(x,y,t),\\
 	%\frac{\partial{u}}{\partial{{n}}} &= 0,\\
 	\begin{split}
 	h(x,y,t) &= e^t (1-x) x (1-y) y+e^t
 	\left(\left(x+\frac{1}{3}\right)^2+\left(y+\frac{1}{4}\right)^2-4\right) \\
 	&\qquad + 2 e^t  (1-x) x+2 e^t (1-y) y,
 	\end{split}
 	\end{align}
 	\label{eq:2D-heat-problem}
 \end{subequations}
with Dirichlet boundary conditions according to the exact solution:
\begin{align}
	u(x,y,t) = & e^t (1-x) x (1-y) y+e^t
	\left(\left(x+\frac{1}{3}\right)^2+\left(y+\frac{1}{4}\right)^2\right).
\end{align}
%%%%%%%%%%%%%%%2D-Conv-Plots%%%%%%%
\begin{figure}[tbhp]
		\centering
		\ref{mylegend2}
		\begin{subfigure}[t]{0.45\textwidth}
			\begin{tikzpicture}
				\begin{loglogaxis}[height=1.8in, grid=major, xlabel={Steps}, ylabel={Error}, , xmin=300, xmax=5500,
				legend entries={$N_p=64\quad$,  $N_p=256\quad$,  $N_p=1024\quad$},
				legend style={at={(0.5, 1.03)},anchor=south}, legend columns=4, legend to name=mylegend2]
					\addplot table [x index=0, y index=1] {Data_v1/2D/ADI_DIMSIM2_2D_S_result.txt}
					 coordinate [pos=0.40] (A)
					 coordinate [pos=0.60] (B);
					\draw (A) |- (B)
					node [anchor = north, pos=0.70] {$m=2$};
					\addplot table [x index=0, y index=2] {Data_v1/2D/ADI_DIMSIM2_2D_S_result.txt};
					\addplot table [x index=0, y index=3] {Data_v1/2D/ADI_DIMSIM2_2D_S_result.txt};
				\end{loglogaxis}
			\end{tikzpicture}
		\caption{\method{2}}
		\label{fig:2D_conv-DIMSIM2}
		\end{subfigure}
		\begin{subfigure}[t]{0.45\textwidth}
			\begin{tikzpicture}
				\begin{loglogaxis}[height=1.8in, grid=major, xlabel={Steps}, ylabel={Error}, xmin=300, xmax=5500]
					\addplot table [x index=0, y index=1] {Data_v1/2D/ADI_DIMSIM3_2D_S_result.txt}
				 coordinate [pos=0.20] (A)
					coordinate [pos=0.40] (B);
					\draw (A) |- (B)
					node [anchor = north, pos=0.7] {$m=3$};
					\addplot table [x index=0, y index=2] {Data_v1/2D/ADI_DIMSIM3_2D_S_result.txt};
					\addplot table [x index=0, y index=3] {Data_v1/2D/ADI_DIMSIM3_2D_S_result.txt};
				\end{loglogaxis}
			\end{tikzpicture}
		\caption{\method{3}}
		\label{fig:2D_conv-DIMSIM3}
		\end{subfigure} 
		\\
		\begin{subfigure}[t]{0.45\textwidth}
			\begin{tikzpicture}
				\begin{loglogaxis}[height=1.8in, grid=major, xlabel={Steps}, ylabel={Error}, xmin=300, xmax=1500]
					\addplot table [x index=0, y index=1] {Data_v1/2D/ADI_DIMSIM4_2D_S_result.txt}
					 coordinate [pos=0.2] (A)
					 coordinate [pos=0.40] (B);
					\draw (A) |- (B)
					node [anchor = north, pos=0.75] {$m=4$};
					\addplot table [x index=0, y index=2] {Data_v1/2D/ADI_DIMSIM4_2D_S_result.txt};
					\addplot table [x index=0, y index=3] {Data_v1/2D/ADI_DIMSIM4_2D_S_result.txt};
				\end{loglogaxis}
			\end{tikzpicture}
		\caption{\method{4}}
		\label{fig:2D_conv-DIMSIM4}
		\end{subfigure} 
		\begin{subfigure}[t]{0.45\textwidth}
			\begin{tikzpicture}
				\begin{loglogaxis}[height=1.8in, grid=major, xlabel={Steps}, ylabel={Error}, xmin=300, xmax=1500,ymax=1e-6,ymin=1e-9]
					\addplot table [x index=0, y index=1] {Data_v1/2D/ADI_RK4_2D_S_result.txt};
					\addplot table [x index=0, y index=2] {Data_v1/2D/ADI_RK4_2D_S_result.txt};
					\addplot table [x index=0, y index=3] {Data_v1/2D/ADI_RK4_2D_S_result.txt}
				 	coordinate [pos=0.6] (A)
					coordinate [pos=0.8] (B);
					\draw (A) -| (B)
					node [anchor = south, pos=0.35] {$m=2$};
				\end{loglogaxis}
			\end{tikzpicture}
		\caption{IMEX-RK4}
		\label{fig:2D_conv-IMEX4}
		\end{subfigure} 
		\caption{Convergence plots for \method{}s on 2D test problem compared to IMEX-RK4 method}
		\label{fig:ADI_2D_conv}
\end{figure}
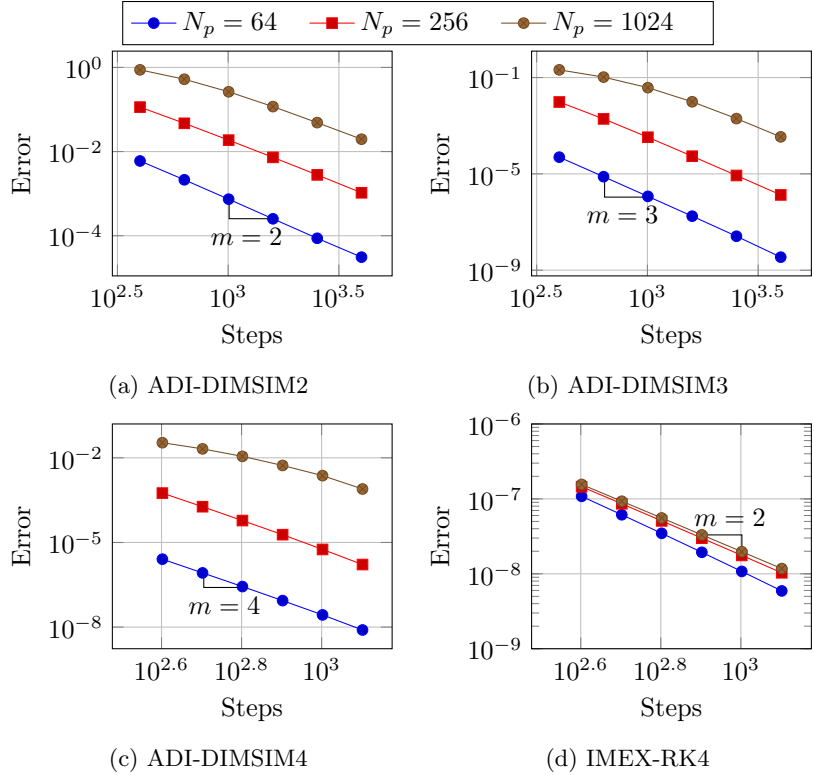
\Cref{fig:ADI_2D_conv} shows convergence plots for this experiment. Once again, we observe the order reduction for the IMEX-RK4 method in \cref{fig:2D_conv-IMEX4} while \method{}s retain their convergence order in \cref{fig:2D_conv-DIMSIM2,fig:2D_conv-DIMSIM3,fig:2D_conv-DIMSIM4}.

%%%%%%%%%%%%%2D-Conv-plots-Explicit-Partition%%%%%%%%%
For a third set of experiments, we examine solutions of \cref{eq:2D-heat-problem}, this time considering the forcing term $g(x,y,t)$ as a third partition to be treated explicitly in the entire integration. This means that the Butcher tableau in \cref{eq:ADI-DIMSIM-2D-3Part} is used for these experiments. \Cref{fig:ADI_3Part_2D_conv} summarizes the results with close to theoretical order of \method{}s.
\begin{figure}[tbhp]
		\centering
		\ref{mylegend3}
		\begin{subfigure}[t]{0.45\textwidth}
			\begin{tikzpicture}
				\begin{loglogaxis}[height=1.8in, grid=major, xlabel={Steps}, ylabel={Error}, xmin=300, xmax=5500,
				legend entries={$N_p=64\quad$,  $N_p=256\quad$,  $N_p=1024\quad$}, legend style={at={(0.5, 1.03)},anchor=south}, legend columns=4, legend to name=mylegend3]
					\addplot table [x index=0, y index=1] {Data_v1/2D_3Partitions/ADI_DIMSIM2_2D_3Part.txt}
					 coordinate [pos=0.2] (A)
					 coordinate [pos=0.4] (B);
					\draw (A) |- (B)
					node [anchor = north, pos=0.75] {$m=2$};
					\addplot table [x index=0, y index=2] {Data_v1/2D_3Partitions/ADI_DIMSIM2_2D_3Part.txt};
					\addplot table [x index=0, y index=3] {Data_v1/2D_3Partitions/ADI_DIMSIM2_2D_3Part.txt};
				\end{loglogaxis}
			\end{tikzpicture}
		\caption{\method{2}}
		\end{subfigure}
		\begin{subfigure}[t]{0.45\textwidth}
			\begin{tikzpicture}
				\begin{loglogaxis}[height=1.8in, grid=major, xlabel={Steps}, ylabel={Error},  xmin=300, xmax=5500]
					\addplot table [x index=0, y index=1] {Data_v1/2D_3Partitions/ADI_DIMSIM3_2D_3Part.txt}
				 coordinate [pos=0.2] (A)
					coordinate [pos=0.4] (B);
					\draw (A) |- (B)
					node [anchor = north, pos=0.7] {$m=3$};
					\addplot table [x index=0, y index=2] {Data_v1/2D_3Partitions/ADI_DIMSIM3_2D_3Part.txt};
					\addplot table [x index=0, y index=3] {Data_v1/2D_3Partitions/ADI_DIMSIM3_2D_3Part.txt};
				\end{loglogaxis}
			\end{tikzpicture}
		\caption{\method{3}}
		\end{subfigure} 
		\\
		 \hspace{.5cm} 
		\begin{subfigure}[t]{0.45\textwidth}
			\begin{tikzpicture}
				\begin{loglogaxis}[height=1.8in, grid=major, xlabel={Steps}, ylabel={Error}, xmin=300, xmax=1500] 
					\addplot table [x index=0, y index=1] {Data_v1/2D_3Partitions/ADI_DIMSIM4_2D_3Part.txt}
					 coordinate [pos=0.2] (A)
					 coordinate [pos=0.4] (B);
					\draw (A) |- (B)
					node [anchor = north, pos=0.7] {$m=4$};
					\addplot table [x index=0, y index=2] {Data_v1/2D_3Partitions/ADI_DIMSIM4_2D_3Part.txt};
					\addplot table [x index=0, y index=3] {Data_v1/2D_3Partitions/ADI_DIMSIM4_2D_3Part.txt};
				\end{loglogaxis}
			\end{tikzpicture}
		\caption{\method{4}}
		\end{subfigure} \hfill
		\begin{subfigure}[t]{0.45\textwidth}
			% \begin{tikzpicture}
			% 	\begin{loglogaxis}[height=1.8in, grid=major, xlabel={Steps}, ylabel={Error}]
			% 		\addplot table [x index=0, y index=1] {Data/2D_3Partitions/ADI_RK4_3D_r2sult.txt}
			% 		 coordinate [pos=0.6] (A)
			% 		 coordinate [pos=0.8] (B);
			% 		\draw (A) |- (B)
			% 		node [anchor = north, pos=0.75] {$m=4$};
			% 		\addplot table [x index=0, y index=2] {Data/2D_3Partitions/ADI_RK4_3D_r2sult.txt};
			% 		\addplot table [x index=0, y index=3] {Data/2D_3Partitions/ADI_RK4_3D_r2sult.txt};
			% 		\addplot table [x index=0, y index=4] {Data/2D_3Partitions/ADI_RK4_3D_r2sult.txt}
			% 		 coordinate [pos=0.65] (C)
			% 		 coordinate [pos=0.85] (D);
			% 		\draw (C) -| (D)
			% 		node [anchor = south, pos=0.3] {$m=2$};
			% 	\end{loglogaxis}
			% \end{tikzpicture}
		%\caption{alaky}
		\end{subfigure} 
		\caption{Convergence plots for \method{}s on 2D test problem with an explicit partition}
		\label{fig:ADI_3Part_2D_conv}
\end{figure}
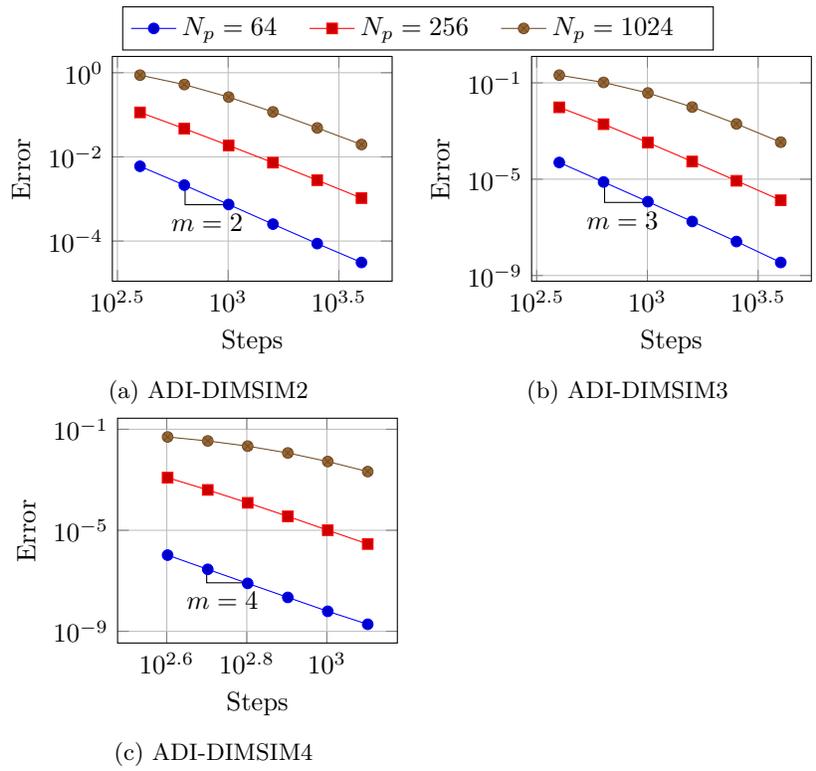
\section{Conclusions}
\label{sec:conclusions}
%%%%%%%%%%%%%%%%%%%

This work constructs the new family of ADI-GLM schemes that perform alternating directions implicit integration in the framework of General Linear Methods. Each stage of a ADI-GLM scheme is implicit in a single component of the method, and is explicitly coupled to the other components. This ensures a high computational efficiency. The ADI character of the method stems from the fact that consecutive stages are implicit in different partitions, thereby ``alternating directions.'' Order conditions and stability of these methods are investigated theoretically. The  ADI-GLM structure allows for high stage order approximations, and this property alleviates the order reduction observed with other families of schemes.

Using the new ADI-GLM theory we construct practical \method{}s of orders two, three, and four. Their design emphasizes  stability when applied to parabolic systems where each component has a Jacobian with real negative eigenvalues. The stability analysis and plots show the new schemes are well-suited for these problems.  Numerical experiments show that the new methods retain their high order of accuracy when applied to parabolic equations with time-dependent Dirichlet boundary conditions where other ADI methods suffer from order reduction.

The future directions for the authors include extending the current set of methodology to design methods suited for hyperbolic and oscillatory systems and numerical experiments highlighting the computational efficiency of \method{}s on large scale problems.

\section*{Acknowledgments}
%%%%%%%%%%%%%%%%%%%
This work was funded by awards NSF CCF--1613905, NSF ACI--1709727, AFOSR DDDAS FA9550-17-1-0015, and by the Computational Science Laboratory at Virginia Tech. The authors would like to thank Prof. Domingo Hern{\'a}ndez Abreu for his valuable comments on this manuscript. 
\section*{References}
\bibliographystyle{elsarticle-num}
\bibliography{Bib/sandu,Bib/glm,Bib/ode_general,Bib/ADI,Bib/pde_time_implicit,Bib/misc,Bib/ode_glm,Bib/ode_splitting}

\begin{thebibliography}{10}
\expandafter\ifx\csname url\endcsname\relax
  \def\url#1{\texttt{#1}}\fi
\expandafter\ifx\csname urlprefix\endcsname\relax\def\urlprefix{URL }\fi
\expandafter\ifx\csname href\endcsname\relax
  \def\href#1#2{#2} \def\path#1{#1}\fi

\bibitem{Douglas_1955_ADI}
J.~Douglas, On the numerical integration of $u_{x,x} + u_{y,y} = u_{t}$ by
  implicit methods, SIAM 3 (1955) 42--65.

\bibitem{Douglas_1956_ADI}
J.~Douglas, H.~H. Rachford, On the numerical solution of heat conduction
  problems in two and three space variables, Transactions of the American
  Mathematical Society 82 (1956) 421--439.

\bibitem{Peaceman_1955_ADI}
D.~Peaceman, H.~Rachford, The numerical solution of parabolic and elliptic
  differential equations, Journal of Society for Indistrial and Applied
  Mathematics 3 (1955) 28--42.

\bibitem{Strang_1968_splitting}
G.~Strang, On the construction and comparison of difference schemes, SIAM
  Journal on Numerical Analysis 5 (1968) 506--517.

\bibitem{Yoshida_1990_splitting}
H.~Yoshida, Construction of higher order symplectic integrators, Physics
  Letters 150 (1990) 262--268.

\bibitem{Yanenko_1971_book}
N.~Yanenko, The Method of Fractional-Steps, Springer, Berlin Heidelberg
  NewYork, 1971.

\bibitem{gonzalez2014rosenbrock}
S.~Gonz{\'a}lez-Pinto, D.~Hern{\'a}ndez-Abreu, S.~P{\'e}rez-Rodr{\'\i}guez,
  Rosenbrock-type methods with inexact {AMF} for the time integration of
  advection--diffusion--reaction {PDE}s, Journal of Computational and Applied
  Mathematics 262 (2014) 304--321.
\newblock \href {http://dx.doi.org/10.1016/j.cam.2013.10.050}
  {\path{doi:10.1016/j.cam.2013.10.050}}.

\bibitem{gonzalez2015amf}
S.~Gonz{\'a}lez-Pinto, D.~Hern{\'a}ndez-Abreu, S.~P{\'e}rez-Rodr{\'\i}guez,
  {AMF}--{R}unge--{K}utta formulas and error estimates for the time integration
  of advection diffusion reaction {PDE}s, Journal of Computational and Applied
  Mathematics 289 (2015) 3--21.
\newblock \href {http://dx.doi.org/10.1016/j.cam.2015.03.048}
  {\path{doi:10.1016/j.cam.2015.03.048}}.

\bibitem{GONZALEZPINTO2018}
S.~Gonz{\'a}lez-Pinto, E.~Hairer, D.~Hern{\'a}ndez-Abreu,
  S.~P{\'e}rez-Rodr{\'\i}guez, {AMF}--type {W}--methods for parabolic problems
  with mixed derivatives, SIAM Journal on Scientific Computing 40~(5) (2018)
  A2905--A2929.
\newblock \href {http://dx.doi.org/10.1137/17M1163050}
  {\path{doi:10.1137/17M1163050}}.

\bibitem{Sandu_2015_AMF-RK}
H.~Zhang, A.~Sandu, P.~Tranquilli, Application of approximate matrix
  factorization to high-order linearly-implicit {Runge-Kutta} methods, Journal
  of Computational and Applied Mathematics 286 (2015) 196--210.
\newblock \href {http://dx.doi.org/10.1016/j.cam.2015.03.005}
  {\path{doi:10.1016/j.cam.2015.03.005}}.

\bibitem{Bujanda_2001_FSRK}
B.~Bujanda, J.~Jorge, Stability results for fractional-step discretizations of
  time dependent coefficient evolutionary problems, Applied Numerical
  Mathematics 38 (2001) 69--86.

\bibitem{Bujanda_2003_FSRK}
B.~Bujanda, J.~Jorge, Fractional-step {Runge--Kutta} methods for time dependent
  coefficient parabolic problems, Applied Numerical Mathematics 45 (2003)
  99--122.

\bibitem{Prothero_1974_PR}
A.~Prothero, A.~Robinson, On the stability and accuracy of one-step methods for
  solving stiff systems of ordinary differential equations, Mathematics of
  Computation 28~(125) (1974) 145--162.

\bibitem{ostermann1992runge}
A.~Ostermann, M.~Roche, {R}unge--{K}utta methods for partial differential
  equations and fractional orders of convergence, Mathematics of computation
  59~(200) (1992) 403--420.

\bibitem{BRAS201794}
M.~Bra{\'s}, A.~Cardone, Z.~Jackiewicz, B.~Welfert, Order reduction phenomenon
  for general linear methods, Applied Numerical Mathematics 119 (2017) 94 --
  114.
\newblock \href {http://dx.doi.org/10.1016/j.apnum.2017.04.001}
  {\path{doi:10.1016/j.apnum.2017.04.001}}.

\bibitem{PORTERO2004409}
L.~Portero, J.~Jorge, B.~Bujanda, Avoiding order reduction of fractional step
  {Runge--Kutta} discretizations for linear time dependent coefficient
  parabolic problems, Applied Numerical Mathematics 48~(3) (2004) 409 -- 424.
\newblock \href {http://dx.doi.org/10.1016/j.apnum.2003.11.006}
  {\path{doi:10.1016/j.apnum.2003.11.006}}.

\bibitem{Jackiewicz_2009_book}
Z.~Jackiewicz, {G}eneral {L}inear {M}ethods for {O}rdinary {D}ifferential
  {E}quations, Wiley, Hoboken, New Jersey, 2009.

\bibitem{Butcher_2001_GLM-stiff}
J.~Butcher, General linear methods for stiff differential equations, BIT 41~(2)
  (2001) 240--264.
\newblock \href {http://dx.doi.org/10.1023/A:1021986222073}
  {\path{doi:10.1023/A:1021986222073}}.

\bibitem{Butcher_2003_construction}
J.~Butcher, W.~Wright, The construction of practical general linear methods,
  BIT 43~(4) (2003) 695--721.
\newblock \href {http://dx.doi.org/10.1023/B:BITN.0000009952.71388.23}
  {\path{doi:10.1023/B:BITN.0000009952.71388.23}}.

\bibitem{Sandu_2014_IMEX-GLM}
H.~Zhang, A.~Sandu, S.~Blaise, Partitioned and implicit-explicit general linear
  methods for ordinary differential equations, Journal of Scientific Computing
  61~(1) (2014) 119--144.
\newblock \href {http://dx.doi.org/10.1007/s10915-014-9819-z}
  {\path{doi:10.1007/s10915-014-9819-z}}.

\bibitem{Sandu_2012_ICCS-IMEX}
H.~Zhang, A.~Sandu, A second-order diagonally-implicit-explicit multi-stage
  integration method, in: Proceedings of the International Conference on
  Computational Science ICCS 2012, Vol.~9, 2012, pp. 1039--1046.
\newblock \href {http://dx.doi.org/10.1016/j.procs.2012.04.112}
  {\path{doi:10.1016/j.procs.2012.04.112}}.

\bibitem{Sandu_2015_Stable_IMEX-GLM}
A.~Cardone, Z.~Jackiewicz, A.~Sandu, H.~Zhang, Construction of highly stable
  implicit-explicit general linear methods, in: AIMS proceedings, Vol.
  Dynamical Systems, Differential Equations, and Applications, Madrid, Spain,
  2015.
\newblock \href {http://dx.doi.org/10.3934/proc.2015.0185}
  {\path{doi:10.3934/proc.2015.0185}}.

\bibitem{Sandu_2016_highOrderIMEX-GLM}
H.~Zhang, A.~Sandu, S.~Blaise, High order implicit--explicit general linear
  methods with optimized stability regions, SIAM Journal on Scientific
  Computing 38~(3) (2016) A1430--A1453.
\newblock \href {http://dx.doi.org/10.1137/15M1018897}
  {\path{doi:10.1137/15M1018897}}.

\bibitem{Sandu_2015_IMEX-TSRK}
E.~Zharovsky, A.~Sandu, H.~Zhang, A class of {IMEX} two-step {Runge-Kutta}
  methods, SIAM Journal on Numerical Analysis 53~(1) (2015) 321--341.
\newblock \href {http://dx.doi.org/10.1137/130937883}
  {\path{doi:10.1137/130937883}}.

\bibitem{Izzo2019}
G.~Izzo, Z.~Jackiewicz, Transformed implicit-explicit {DIMSIMs} with strong
  stability preserving explicit part, Numerical Algorithms 81~(4) (2019)
  1343--1359.
\newblock \href {http://dx.doi.org/10.1007/s11075-018-0647-3}
  {\path{doi:10.1007/s11075-018-0647-3}}.

\bibitem{SCHNEIDER2018121}
M.~Schneider, J.~Lang, W.~Hundsdorfer, Extrapolation--based super-convergent
  implicit--explicit {Peer} methods with {A--stable} implicit part, Journal of
  Computational Physics 367 (2018) 121 -- 133.
\newblock \href {http://dx.doi.org/10.1016/j.jcp.2018.04.006}
  {\path{doi:10.1016/j.jcp.2018.04.006}}.

\bibitem{schneider2019super}
M.~Schneider, J.~Lang, R.~Weiner, Super-convergent implicit-explicit peer
  methods with variable step sizes, arXiv preprint arXiv:1902.01161.

\bibitem{zhang2016high}
H.~Zhang, A.~Sandu, S.~Blaise, High order implicit-explicit general linear
  methods with optimized stability regions, SIAM Journal on Scientific
  Computing 38~(3) (2016) A1430--A1453.

\bibitem{Butcher_1993_diagonal}
J.~Butcher, Z.~Jackiewicz, Diagonally implicit general linear methods for
  ordinary differential equations, BIT 33~(3) (1993) 452--472.
\newblock \href {http://dx.doi.org/10.1007/BF01990528}
  {\path{doi:10.1007/BF01990528}}.

\bibitem{jackiewicz2009general}
Z.~Jackiewicz, General linear methods for ordinary differential equations, John
  Wiley \& Sons, 2009.

\bibitem{butcher1993diagonally}
J.~Butcher, Z.~Jackiewicz, Diagonally implicit general linear methods for
  ordinary differential equations, BIT Numerical Mathematics 33~(3) (1993)
  452--472.

\bibitem{califano2017starting}
G.~Califano, G.~Izzo, Z.~Jackiewicz, Starting procedures for general linear
  methods, Applied Numerical Mathematics 120 (2017) 165--175.

\bibitem{Sandu_2015_GARK}
A.~Sandu, M.~G\"{u}nther, A generalized-structure approach to additive
  {Runge-Kutta} methods, SIAM Journal on Numerical Analysis 53~(1) (2015)
  17--42.
\newblock \href {http://dx.doi.org/10.1137/130943224}
  {\path{doi:10.1137/130943224}}.

\bibitem{GLM-ADI-Coeffs}
A.~Sarshar, S.~Roberts, A.~Sandu, {ADI-GLM} coefficients”, Mendeley Data
  (2019).
\newblock \href {http://dx.doi.org/10.17632/cxnhv3m2sx.2}
  {\path{doi:10.17632/cxnhv3m2sx.2}}.

\bibitem{Butcher_1993_DIMSIM}
J.~Butcher, Diagonally-implicit multi-stage integration methods, Applied
  Numerical Mathematics 11~(5) (1993) 347--363.
\newblock \href {http://dx.doi.org/10.1016/0168-9274(93)90059-Z}
  {\path{doi:10.1016/0168-9274(93)90059-Z}}.

\bibitem{Butcher_1996_construction}
J.~Butcher, Z.~Jackiewicz, Construction of diagonally implicit general linear
  methods of type 1 and 2 for ordinary differential equations, Applied
  Numerical Mathematics 21~(4) (1996) 385--415.
\newblock \href {http://dx.doi.org/10.1016/S0168-9274(96)00043-8}
  {\path{doi:10.1016/S0168-9274(96)00043-8}}.

\end{thebibliography}
\appendix
\section{ADI-GLMs}
\label{app:methods}
This section includes the newly developed \method{}s of orders two, three, and four. MATLAB files containing these coefficients are also available in [dataset] \cite{GLM-ADI-Coeffs}.
\begingroup
\renewcommand*{\arraystretch}{1.4}
%\allowdisplaybreaks

\subsection{\method{2}}
\label{app:method2}

We use an L-stable implicit base method from \cite{Butcher_1993_DIMSIM} for \method{2}.

\begin{alignat*}{2}
	\A\comp{I} &= \begin{bmatrix}
		\frac{2 - \sqrt{2}}{2} & 0 \\
		\frac{2 (\sqrt{2}+3)}{7} & \frac{2 - \sqrt{2}}{2} \\
	\end{bmatrix}, \quad
	& \B\comp{I} &= \begin{bmatrix}
		\frac{73-34 \sqrt{2}}{28} & \frac{4 \sqrt{2} - 5}{4} \\
		\frac{3(29 - 16 \sqrt{2})}{28} & \frac{34 \sqrt{2}-45}{28} \\
	\end{bmatrix}, \\
	\W\comp{I} &= \begin{bmatrix}
		1 & \frac{\sqrt{2} - 2}{2} & 0 \\
		1 & \frac{3 (\sqrt{2}-4)}{14} & \frac{\sqrt{2} - 1}{2} \\
	\end{bmatrix}, \quad
	& \A\comp{E} &= \begin{bmatrix}
		0 & 0 \\
		\frac{3}{2} & 0 \\
	\end{bmatrix}, \\
	\B\comp{E} &= \begin{bmatrix}
		\frac{1}{\sqrt{2}} & \frac{3-\sqrt{2}}{4} \\
		\frac{\sqrt{2}-1}{2} & \frac{3-\sqrt{2}}{4} \\
	\end{bmatrix}, \quad
	& \W\comp{E} &= \begin{bmatrix}
		1 & 0 & 0 \\
		1 & -\frac{1}{2} & \frac{1}{2} \\
	\end{bmatrix}, \\
	v &= \begin{bmatrix}
		\frac{3 - \sqrt{2}}{2} & \frac{\sqrt{2} - 1}{2}
	\end{bmatrix}^T, \quad
	& \c &= \begin{bmatrix}
		0 & 1
	\end{bmatrix}^T.
\end{alignat*}

\subsection{\method{3}}

We use an L-stable implicit base method from \cite{Butcher_1996_construction} for \method{3}.  The following coefficients are rational approximations to the exact coefficients accurate to 24 digits.

\begin{align*}
	\A\comp{I} &= \begin{bmatrix}
		\frac{129981159316}{298213221025} & 0 & 0 \\
		\frac{472981046840}{1888035733227} & \frac{129981159316}{298213221025} & 0 \\
		-\frac{408860438935}{337456558734} & \frac{1049716501919}{1048380236594} & \frac{129981159316}{298213221025} \\
	\end{bmatrix}, \\
	\B\comp{I} &= \begin{bmatrix}
		\frac{818629988268}{981817092145} & \frac{735879558291}{1139134361459} & -\frac{96693387431}{306159262034} \\
		\frac{435713380671}{718693545019} & \frac{3397277300866}{2639826970205} & -\frac{581689679739}{1212506039656} \\
		-\frac{164008995335}{531777165056} & \frac{3204278525979}{842472621931} & -\frac{1170634530631}{1044535547981} \\
	\end{bmatrix}, \\
	\W\comp{I} &= \begin{bmatrix}
		1 & -\frac{129981159316}{298213221025} & 0 & 0 \\
		1 & -\frac{63231801579}{339260252164} & -\frac{94226735668}{1013918320559} & -\frac{50172116077}{1490999795865} \\
		1 & \frac{1224205243956}{1580735023225} & -\frac{377260820095}{864278390147} & -\frac{145496067686}{824686465859} \\
	\end{bmatrix}, \\
	\A\comp{E} &= \begin{bmatrix}
		0 & 0 & 0 \\
		\frac{692830401049}{1119419041371} & 0 & 0 \\
		-\frac{974910195245}{1036334372568} & \frac{1458124485343}{1218848111125} & 0 \\
	\end{bmatrix}, \\
	\B\comp{E} &= \begin{bmatrix}
		\frac{274198327012}{348784765929} & \frac{335124252337}{1242427076379} & \frac{256046237035}{1044616400532} \\
		\frac{2367946890051}{2381074405894} & -\frac{395462379375}{996294720374} & \frac{391448928279}{669688356392} \\
		\frac{1211513153203}{1601457627995} & \frac{473388990672}{901108379101} & \frac{1335987676745}{1749669440649} \\
	\end{bmatrix}, \\
	\W\comp{E} &= \begin{bmatrix}
		1 & 0 & 0 & 0 \\
		1 & -\frac{105007291910}{883010702197} & \frac{1}{8} & \frac{1}{48} \\
		1 & \frac{6500435948486}{8732264247243} & -\frac{119638187109}{1218848111125} & \frac{25266119777}{1475180609484} \\
	\end{bmatrix}, \\
	v &= \begin{bmatrix}
		\frac{1611220452657}{2918396719813} & \frac{626900045900}{853091602939} & -\frac{165394139815}{576391394057}
	\end{bmatrix}^T, \\
	\c &= \begin{bmatrix}
		0 & \frac{1}{2} & 1
	\end{bmatrix}^T.
\end{align*}

\subsection{\method{4}}

\begin{sidewaystable}
\begin{align*}
	\A\comp{I} &= \begin{bmatrix}
		\frac{2}{5} & 0 & 0 & 0 \\
		\frac{1}{155} & \frac{2}{5} & 0 & 0 \\
		-\frac{3}{127} & \frac{31}{72} & \frac{2}{5} & 0 \\
		\frac{6}{139} & \frac{12}{19} & \frac{29}{95} & \frac{2}{5} \\
	\end{bmatrix}, \quad
	\B\comp{I} = \begin{bmatrix}
		\frac{25640275033859}{233564187988800} & \frac{405169687}{540615360} & \frac{1089772729}{8109230400} & \frac{70445177}{426801600} \\
		\frac{89870426730779}{233564187988800} & -\frac{545995987}{1621846080} & \frac{13906861889}{8109230400} & -\frac{1223893451}{4410283200} \\
		\frac{292292722987739}{233564187988800} & -\frac{5722388059}{1621846080} & \frac{5251926081}{901025600} & -\frac{115646334041}{54203803200} \\
		\frac{12591629268162881}{4437719571787200} & -\frac{4936252337}{540615360} & \frac{102615203329}{8109230400} & -\frac{5841129112303}{1127183025600} \\
	\end{bmatrix}, \\
	\W\comp{I} &= \begin{bmatrix}
		1 & -\frac{2}{5} & 0 & 0 & 0 \\
		1 & -\frac{34}{465} & -\frac{7}{90} & -\frac{13}{810} & -\frac{19}{9720} \\
		1 & -\frac{6413}{45720} & -\frac{203}{1080} & -\frac{137}{2160} & -\frac{827}{58320} \\
		1 & -\frac{5018}{13205} & -\frac{179}{570} & -\frac{233}{1710} & -\frac{2707}{61560} \\
	\end{bmatrix}, \quad
	\A\comp{E} = \begin{bmatrix}
		0 & 0 & 0 & 0 \\
		\frac{768}{7129} & 0 & 0 & 0 \\
		\frac{2699}{8714} & \frac{4969}{11444} & 0 & 0 \\
		\frac{2629}{3049} & \frac{2643}{20780} & \frac{11707}{22938} & 0 \\
	\end{bmatrix}, \\
	\B\comp{E} &= \begin{bmatrix}
		\frac{9887514441977875393}{8084061960608111040} & \frac{75125403707867}{1268701473326400} & \frac{200041286909}{326332503360} & -\frac{5924747}{85360320} \\
		\frac{8877006696901861513}{8084061960608111040} & \frac{727096994167267}{1268701473326400} & -\frac{67370070011}{326332503360} & \frac{119019300359}{202844573760} \\
		\frac{17771936994130966829533}{23128501269299805685440} & \frac{249067742877763}{140966830369600} & -\frac{519937182674317}{311212430704320} & \frac{940676971064501}{1064048569640640} \\
		\frac{8566493244911672759404729}{32110678261545596037650880} & \frac{17071987325364576461}{4850245732526827200} & -\frac{257098496412689}{67811894198208} & \frac{275159340062707361}{206758465410336192} \\
	\end{bmatrix}, \\
	\W\comp{E} &= \begin{bmatrix}
		1 & 0 & 0 & 0 & 0 \\
		1 & \frac{4825}{21387} & \frac{1}{18} & \frac{1}{162} & \frac{1}{1944} \\
		1 & -\frac{11557817}{149584524} & \frac{7981}{102996} & \frac{46831}{1853928} & \frac{30869}{5561784} \\
		1 & -\frac{363193513153}{726655425180} & \frac{83904703}{714977460} & \frac{198121973}{4289864760} & \frac{302650439}{19304391420} \\
	\end{bmatrix}, \\
	v &= \begin{bmatrix}
		\frac{3}{40} & -\frac{77}{277} & -\frac{41}{107} & \frac{1880483}{1185560} \\
	\end{bmatrix}^T, \quad
	\c = \begin{bmatrix}
		0 & \frac{1}{3} & \frac{2}{3} & 1
	\end{bmatrix}^T.
\end{align*}
\end{sidewaystable}

\endgroup
\section{Stability of ADI-GLMs}
\label{app:stability-plots}
%%%
\begin{figure}[ht!]
	\begin{subfigure}{0.48 \textwidth}
		\includegraphics[width=\textwidth]{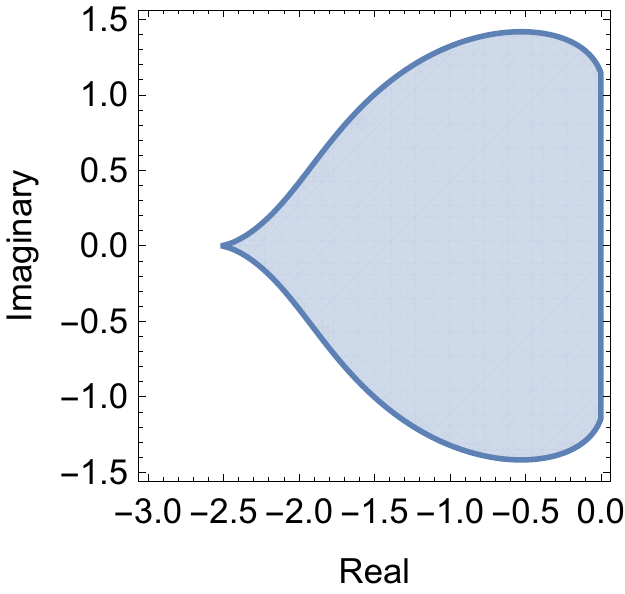}
		\caption{$\mathcal{S}\comp{E}$ stability region}
	\end{subfigure} \hfill
	\begin{subfigure}{0.48 \textwidth}
		\includegraphics[width=\textwidth]{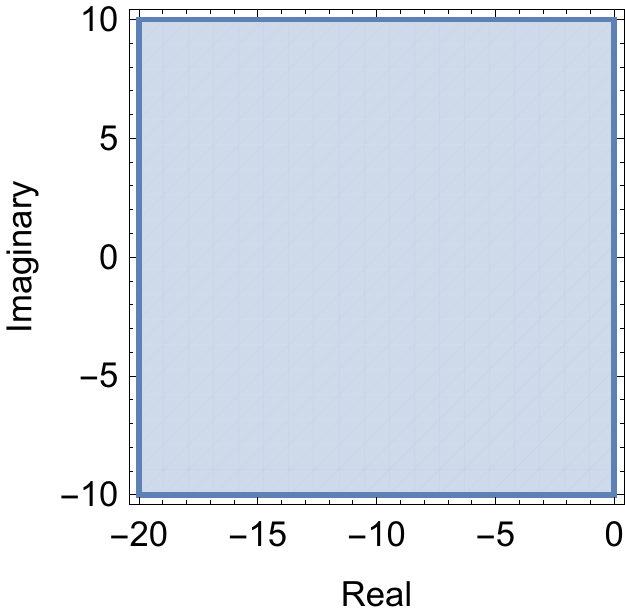}
		\caption{$\mathcal{S}\comp{I}$ stability region}
	\end{subfigure}

	\begin{subfigure}{0.48 \textwidth}
		\includegraphics[width=\textwidth]{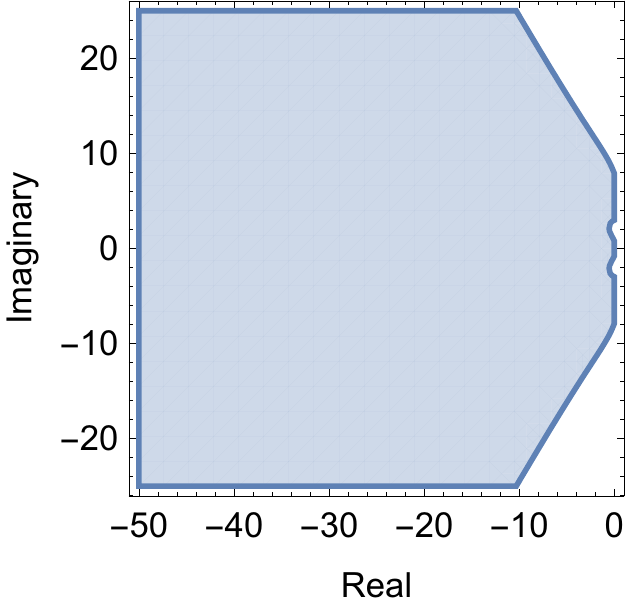}
		\caption{$\mathcal{S}_{\text{Cplx}}$ stability region with $\alpha = \ang{60}$}
	\end{subfigure} \hfill
	\begin{subfigure}{0.48 \textwidth}
		\includegraphics[width=\textwidth]{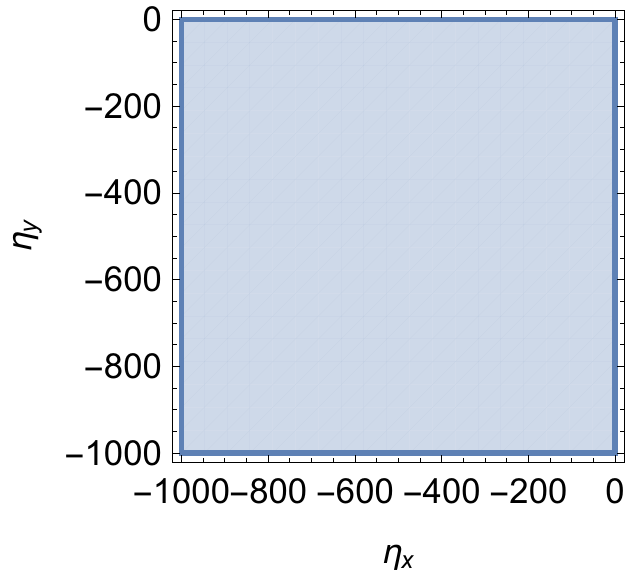}
		\caption{$\mathcal{S}_{\text{Real}}$ stability region}
	\end{subfigure}
\caption{Stabilty plots for \method{2}}
\label{fig:stab-plot-DIMSIM2}
\end{figure}
\begin{figure}[ht!]
	\begin{subfigure}{0.48 \textwidth}
		 \includegraphics[width=\textwidth]{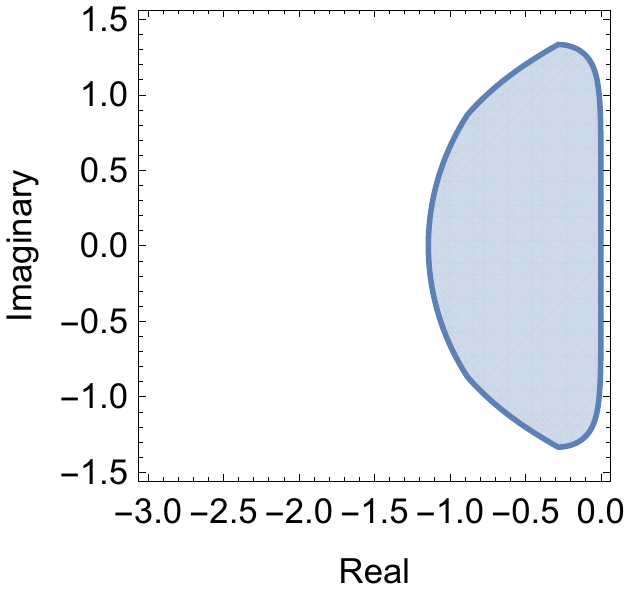}
		\caption{$\mathcal{S}\comp{E}$ stability region}
	\end{subfigure} \hfill
	\begin{subfigure}{0.48 \textwidth}
		 \includegraphics[width=\textwidth]{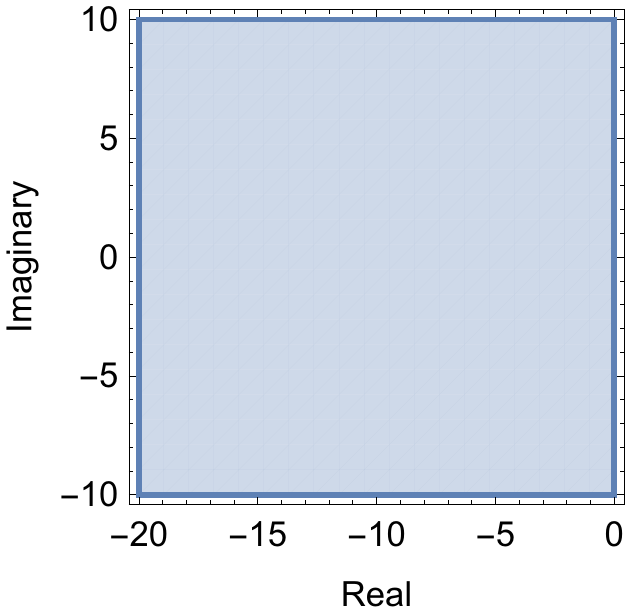}
		\caption{$\mathcal{S}\comp{I}$ stability region}
	\end{subfigure}
	
	\begin{subfigure}{0.48 \textwidth}
		 \includegraphics[width=\textwidth]{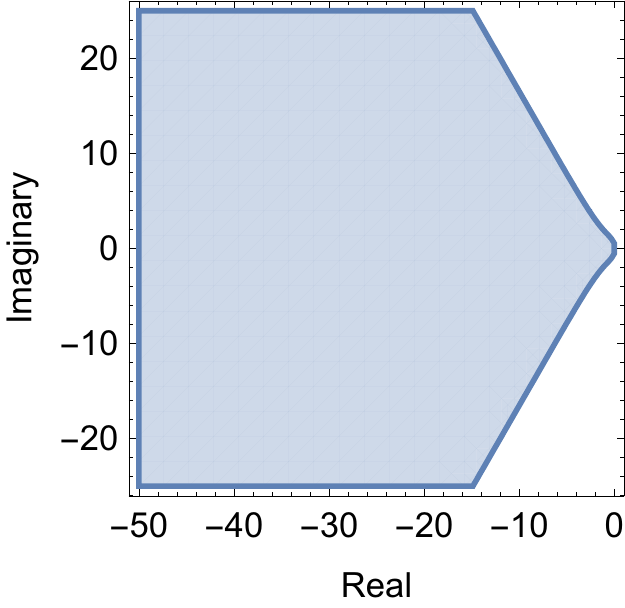}
		\caption{$	\mathcal{S}_{\text{Cplx}}$ stability region with $\alpha = \ang{55}$}
	\end{subfigure} \hfill
	\begin{subfigure}{0.48 \textwidth}
		 \includegraphics[width=\textwidth]{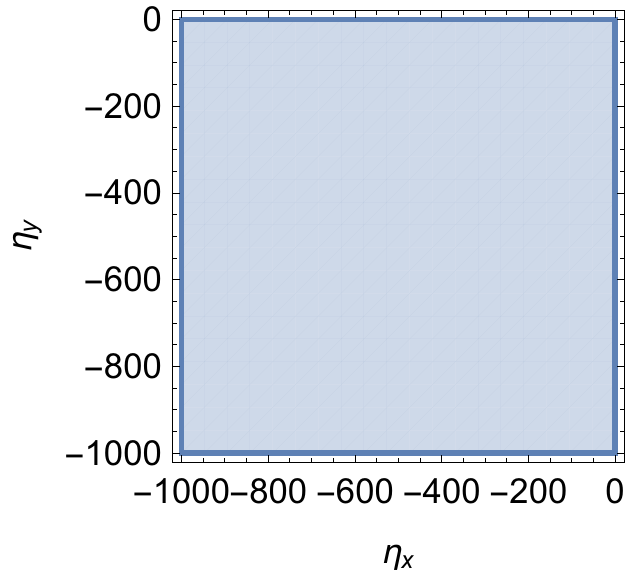}
		\caption{$\mathcal{S}_{\text{Real}}$ stability region}
	\end{subfigure}
	\caption{Stabilty plots for \method{3}}
	\label{fig:stab-plot-DIMSIM3}
\end{figure}
\begin{figure}[ht!]
	\begin{subfigure}{0.48 \textwidth}
		 \includegraphics[width=\textwidth]{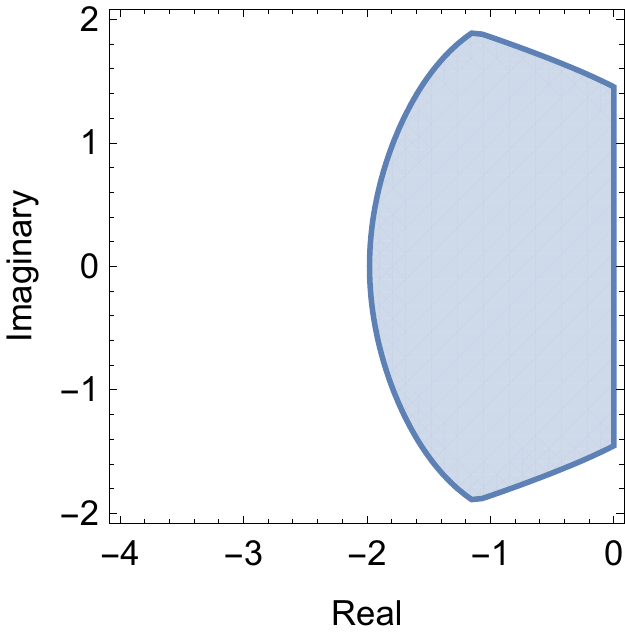}
		\caption{$\mathcal{S}\comp{E}$ stability region}
	\end{subfigure} \hfill
	\begin{subfigure}{0.48 \textwidth}
		 \includegraphics[width=\textwidth]{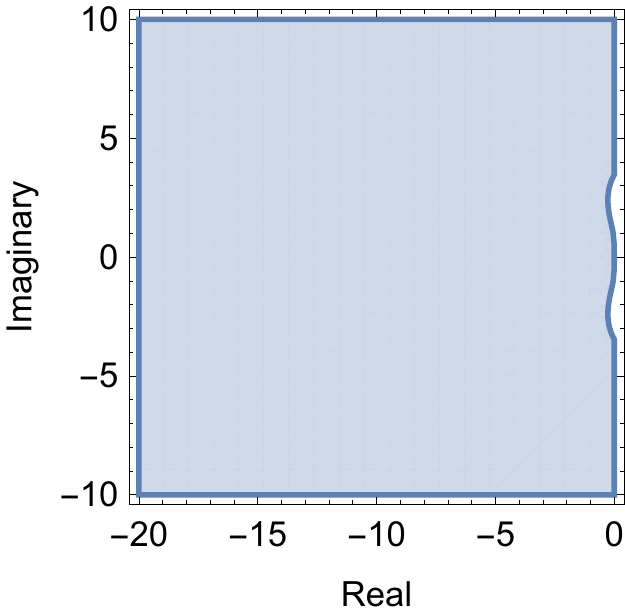}
		\caption{$\mathcal{S}\comp{I}$ stability region with $\alpha = \ang{83}$}
	\end{subfigure}
	
	\begin{subfigure}{0.48 \textwidth}
		 \includegraphics[width=\textwidth]{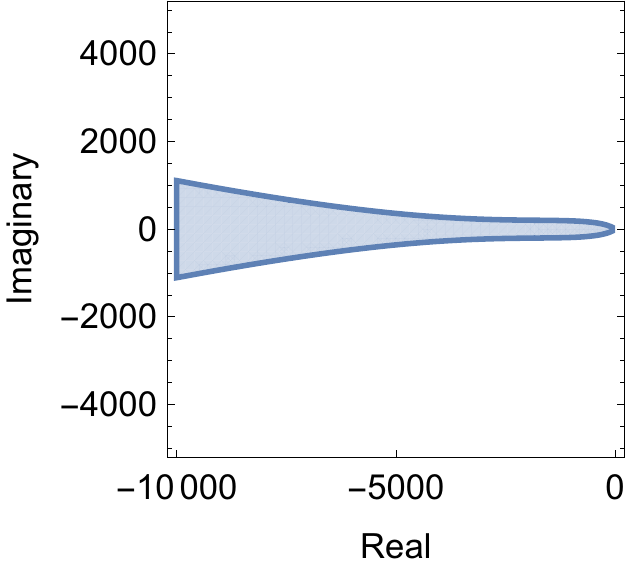}
		\caption{$\mathcal{S}_{\text{Cplx}}$ with $\alpha = \ang{3.7}$}
	\end{subfigure} \hfill
	\begin{subfigure}{0.48 \textwidth}
		 \includegraphics[width=\textwidth]{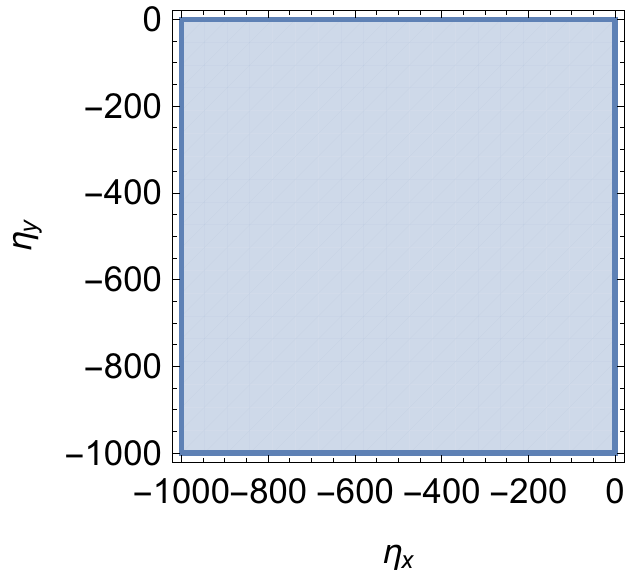}
		\caption{$	\mathcal{S}_{\text{Real}}$}
	\end{subfigure}
	\caption{Stability plots for \method{4}}
	\label{fig:stab-plot-DIMSIM4}
\end{figure}
\end{document}